\documentclass[11pt]{amsart}
\usepackage{amssymb}
\usepackage[margin=1in]{geometry}
\usepackage{enumerate,ulem}
\usepackage[pdftex,bookmarks=true]{hyperref}
\usepackage{fancyhdr,color}

\newcommand\supp{\operatorname{supp}}

\renewcommand\Re{\operatorname{Re}}
\renewcommand\Im{\operatorname{Im}}

\theoremstyle{plain}
\newtheorem{theorem}{Theorem}[section]
\newtheorem{corollary}[theorem]{Corollary}
\newtheorem{lemma}[theorem]{Lemma}
\newtheorem{proposition}[theorem]{Proposition}
\theoremstyle{definition}
\newtheorem{remark}[theorem]{Remark}
\newcommand{\E}{\mathbb E }
\renewcommand{\P}{\mathbb P }
\newcommand{\R}{\mathbb R}
\newcommand{\C}{\mathbb C}
\newcommand{\ep}{\varepsilon}
\newcommand{\K}{\mathcal K}

\begin{document}
	
\title[Random orthonormal polynomials]{Random orthonormal polynomials: local universality and expected number of real roots}
\author{Yen Do}
\address{Department of Mathematics, The University of Virginia, Charlottesville, VA 22904, USA}
\email{yendo@virginia.edu}

\author{OANH NGUYEN}
\address{Division of Applied Mathematics, Brown University, Providence RI 02906, USA}
\email{oanh\_nguyen1@brown.edu}

\author{VAN VU}
\address{Department of Mathematics, Yale University, New Haven, CT 06520, USA}
\email{van.vu@yale.edu}

\thanks{Y. Do partially supported by NSF grant DMS-1800855, O. Nguyen supported by NSF grant DMS-1954174, V. Vu supported by NSF grant DMS-1902825.}
\date{}

\begin{abstract}We consider random orthonormal polynomials
$$
F_{n}(x)=\sum_{i=0}^{n}\xi_{i}p_{i}(x),
$$
where $\xi_{0}$, \dots, $\xi_{n}$ are independent random variables with zero mean, unit variance and uniformly bounded $(2+\ep)$ moments, and $(p_n)_{n=0}^{\infty}$ is the system of orthonormal polynomials with respect to a fixed compactly supported measure on the real line. 

Under mild technical assumptions satisfied by many classes of classical  polynomial systems, we establish universality for the leading asymptotics of the average number of real roots of $F_n$, both globally and locally. 

Prior to this paper, these results were known only for random orthonormal polynomials with Gaussian coefficients \cite{lubinsky2016linear} using the Kac-Rice formula, a method that does not extend to the generality of our paper. 

	
\end{abstract}

\maketitle

\section{Introduction}

One of the main goals in analysis (or mathematics in general) is to study real roots of real functions. Functions are often expanded using a basis -- in other words, as linear combination of  selected functions $p_0, p_1, p_2, p_3, \dots$ 
$$F_n(x)= c_0 p_0(x)+ c_1 p_1(x) +c_2 p_2(x)+\dots+c_n p_n(x).$$
Here, $c_i$ are real coefficients. The most popular bases are:
\begin{itemize}
\item The canonical base: $1, x, x^2,\dots$  This leads to the Taylor expansion.
\item Fourier base: $1$, $\sin(x), \cos (x), \sin(2x), \cos (2x),\dots$ This leads to the Fourier expansion. 
\item Orthogonal base: $p_0, p_1, p_2, p_3, \dots$ are orthonormal polynomials. 
\end{itemize}

Starting with fundamental  works of Littlewood-Offord \cite{LO1,LO2,LO3} and Kac \cite{Kac1943average}  in the 1940s, many researchers studied 
roots of random functions.  To generate a random function, we let $\xi_0,..., \xi_n$ be  iid random variables with mean $0$ and variance $1$, and consider 
$$F_n(x) =\xi_0 p_0(x)  +...+ \xi_n p_n(x).$$
The number of real roots of $F_n$ now becomes a random variable, which we denote by $N_n$.  This object has been studied intensively through many decades. 

For the canonical case with $p_i(x) =x^i$, we have  $F_n(x) =\xi_0 +\dots+ \xi_n x^n$. This is called Kac polynomial. 
In a classical result \cite{Kac1943average}, Kac showed that

\begin{theorem}[Kac, 1943]
Consider the Kac polynomial $F_n(x) =\xi_0 +\dots+ \xi_n x^n$ where $\xi_j$ are iid standard Gaussian. Then
$$\E N_n = \frac {2}{\pi}\log n + o(\log n).$$
\end{theorem}
More generally, Kac also obtained a  formula for $\E N_n$, as follows
$$ \E N_n=\int_{-\infty}^\infty dt \int_{-\infty}^\infty |y|p(t,0,y) dy$$
where $p(t,x,y)$ is the joint probability density for $f(t)=x$ and $f'(t)=y$ (when the $\xi_i$ are not smooth, this is replaced by the joint distribution measure). In the Gaussian case, it is easy to evaluate this density explicitly, which was used by Kac to derive his famous formula.

It is however very hard to evaluate Kac's formula for other distributions and new ideas were required to  extend this result to non-smooth distributions. A new approach was developed by Erd\'os--Offord \cite{EO} to handle discrete distributions (such as the Rademacher distribution where $\xi_j=\pm 1$ with probability $1/2$ each). This method was later refined by Ibragimov--Maslova \cite{Ibragimov1971expected1} to prove

\begin{theorem} [Ibragimov--Maslova, 1971] Consider the Kac polynomial $F_n(x) =\xi_0 +\dots+ \xi_n x^n$. Assume that $\xi_j$ are iid with zero mean and belong to the domain of attraction of the normal law. Then
$$\E N_n = \frac 2{\pi}\log n + o(\log n).$$
\end{theorem}
We note that the estimate $o(\log n)$ for the error term  is  not sharp and has been improved to $O(1)$ in \cite{NNV}, even to $C + o(1)$ for some classical distributions \cite{DHV}.\footnote{The constant $C$  may depend on the distribution $\xi$ of the coefficients $\xi_j$.}

Next, for random polynomials with the trigonometric basis. Let $N_n$ is the number of real roots of for the random trigonometric polynomial 
$$F_n(x)=\xi_{11} \cos(t)+ \xi_{21} \sin(t) + \dots + \xi_{1n}\cos(nt)+\xi_{2n}\sin(nt),$$
where $(\xi_{ik})_{1\le i\le 2, 1\le k\le n}$ are iid standard Gaussian, then Qualls \cite{Q} showed that
$$\E N_n[-\pi,\pi] =2\sqrt {(2n+1)(n+1)/6} = \frac 2{\sqrt 3}n +o(n).$$
The asymptotics for $\E N_n[-\pi,\pi]$ has been proved for more general distributions, such as $\xi_j$ independent with zero mean and unit variance \cite{flasche, nguyenvurandomfunction17}.

Much less is known for orthonormal bases.  Let $\mu$ be a finite positive Borel measure on the real line with infinitely many points
in its support, and with finite power moments of all orders. This guarantees the existence of the system of orthonormal polynomials $(p_n)$ with respect to $\mu$; that is  $p_n$ has degree $n$, real coefficients, positive leading coefficient, and furthermore
$$\int_{\R} p_n(x)p_m(x)d\mu(x) = \begin{cases} 1, & m=n;\\ 0, & m\ne n.\end{cases}$$
These polynomials can be computed using standard arguments as all moments of $\mu$ are finite, (ref. Freud \cite{freudorthogonal} for instance). 

The corresponding combination, which is known as random orthogonal polynomial, is 
\begin{equation}\label{def:F}
F_n(x)=\sum_{j=0}^n \xi_j p_j(x).
\end{equation}

For technical reasons,  we shall assume furthermore that $\mu$ is absolutely continuous with respect to the Lebesgue measure, i.e. $d\mu(x)=w(x)dx$. Some popular choices are 

\begin{itemize}
\item Legendre polynomials: here the weight $w(x)$ is the indicator function of $[-1,1]$.\\
\item Chebyshev polynomials of type 1: here $w(x)=\frac 1{\sqrt{1-x^2}}$ for $x\in (-1,1)$, and zero elsewhere.\\
\item Chebyshev polynomials of type 2: here $w(x)=\sqrt{1-x^2}$ for $x\in (-1,1)$, and zero elsewhere.\\
\item Jacobi polynomials: these are generalizations of Chebyshev and Legendre polynomials, with $w(x)=(1-x)^\beta (1+x)^\gamma$ for $x\in (-1,1)$ and zero elsewhere. The parameters $\beta,\gamma$ are given constants larger than $-1$.\\
\item Laguerre polynomials: here $w(x)=e^{-x}$ for  $x\in \R$.\\
\item Hermitte polynomials: here $w(x)=e^{-x^2}$ for $x\in \R$.
\end{itemize}

The analogue of Kac's theorem for these classes was obtained  \cite{das1971real, wilkins1997expected,das1982real,lubinsky2016linear} for Gaussian distribution.   In particular,  
\begin{itemize}
\item random Legendre polynomials were considered in Das \cite{das1971real} and Wilkins \cite{wilkins1997expected};
\item random Jacobi polynomials, random Hermite polynomials, and random Laguerre polynomials were considered  in Das--Bhatt \cite{das1982real}; 
\end{itemize} 
Lubinsky, Pritsker and Xie  \cite{lubinsky2016linear} showed a very interesting result that for general orthogonal polynomials, the asymptotics for the number of real roots does not depend on the actual measure $\mu$ but only on its support. Their method, using Kac-Rice formula, replied heavily on the assumption that the random variables are standard Gaussian.

\begin{theorem}[Lubinsky--Pritsker--Xie, 2016]\label{t.lpx}
Let $\K\subset \mathbb{R}$ be a finite union of compact intervals. Let $\mu$  be a  Borel measure supported on $\K$ with density $w(x)$ such that $w>0$ for a.~e. $x\in \K$. 

Assume that for every $\ep>0$ there is a constant $C=C(\ep)>1$ and a closed set $S_{\ep}$  of Lebesgue measure less than $\ep$  such that $C^{-1}<w<C$ a.e.  on $\K\backslash S_{\ep}$.

Let $N_n$ be the number of real roots of $F_n(x)=\sum_{j=0}^n \xi_j p_j(x)$ where $\xi_j$ are iid standard Gaussian and $(p_n)$ is the system of orthonormal polynomials with respect to $\mu$. Then 
\begin{equation} \label{eq2}
\E N_n=\frac{1}{\sqrt{3}}n + o(n).
\end{equation}
Furthermore, for any fixed compact interval $[a,b]$ inside $\K\setminus S_{\ep}$ for some $\ep>0$, it holds that
\begin{equation} \label{eq4}
\E N_n([a, b]) =\frac{\nu_\K ([a,b])}{\sqrt{3}} n + o(n),
\end{equation}
where $\nu_{\K}$ is the equilibrium measure of $\K.$  
\end{theorem}

We recall that the equilibrium measure  $\nu_\K$   minimizes 
$$
I[\nu]=-\iint\log|z-t|d\nu(t)d\nu(z)
$$
among all probability measures $\nu$   supported on $\K$. It is well-known that $\nu_{\K}$ exists, is unique and absolutely continuous with respect to the Lebesgue measure on $\K$ (see  for instance Stahl--Totik \cite[Lemma 4.4.1]{stahl1992general}).

It is natural to ask whether this result remains valid when the $\xi_j$ are no longer Gaussian. This is the goal of our paper. We will assume that the random variables $\xi_j$ are independent (but not necessarily identically distributed) and have uniformly bounded $(2+\ep)$ moments for some $\ep>0$. In other words, there exist positive constants $\ep, C$, independent of $n$, such that for all $0\le j\le n$,
$$\E |\xi_j|^{2+\ep}<C.$$ 

Here is our main result. To the best of our knowledge, this is the first known result for real roots of random polynomials with non-Gaussian coefficients. 
\begin{theorem}\label{t.main1} Let $\mu$ be a measure supported in a finite union $\K$ of compact intervals in $\R$ with a density $w(x)$ with respect to the Lebesgue measure and the orthonormal system $(p_n)_{n=0}^{\infty}$. Assume that for every compact interval $I$ in the interior of $\K$,  there exists a constant $C_I>0$ for which the following hold
		\begin{equation}\label{cond:w} 
	C_I^{-1}\le w(x)\le C_I\quad \text{	almost everywhere on $I$,}
	\end{equation}
and
	\begin{equation}\label{cond:pn}
\sup_{n\ge 0} \sup_{x\in I} |p_n(x)|\le C_I.
	\end{equation}

Let $\xi_j$ be independent random variables with zero mean and unit variance and uniformly bounded $(2+\ep)$ moments, for some fixed $\ep>0$.
Let $N_n(S)$ be the number of roots of $F_n(x) = \sum_{j=0}^n \xi_j p_j(x)$ in a set $S$.
Then there exists a constant $c>0$ such that for any compact interval $[a,b]$ in the interior of $\K$, the following holds
	\begin{equation} \label{eq:main1:2}
	 \E N_n([a,b]) = \frac    {\nu_K([a,b])}{\sqrt 3}n + o(n^{1-c}).
	\end{equation}
\end{theorem}

\begin{theorem} \label{thm:real} Under the hypothesis of Theorem \ref{t.main1}, assume furthermore that the $\xi_i$ are identically distributed, then
\begin{equation} \label{eq:main1:1}
	\E N_n(\R ) = \frac 1 {\sqrt 3} n + o(n).
	\end{equation}
\end{theorem}
\begin{remark}
	The extra assumption that the random variables are identically distributed allows us to directly apply Theorem \ref{thm:conv:nu}. We think that this condition can be removed, given the conditions of $\xi_i$ in Theorem \ref{t.main1}.
\end{remark}



As an application, we apply Theorems \ref{t.main1} and \ref{thm:real} to the system of classical Jacobi polynomials $\left (p_n^{(\beta,\gamma)}\right )_{n\ge 0}$ (which by using particular values of $(\beta,\gamma)$ covers the Legendre polynomials, the Chebyshev polynomials of type 1, and  the Chebyshev polynomials of  type 2). Recall that the measure generating this system  is supported on $(-1, 1)$ and has the following density
\begin{equation}\label{jacobi:w}
w_{\beta, \gamma} = (1-x)^{\beta}(1+x)^{\gamma},\quad\forall x\in (-1, 1).
\end{equation}
It is clear that this density satisfies the hypothesis  involving \eqref{cond:w} in Theorem~\ref{t.main1}.

The following classical theorem from \cite{freudorthogonal} provides an upper bound for $p_n^{(\beta, \gamma)}$.
\begin{lemma} \cite[Theorem 8.1, page 45]{freudorthogonal}
	For every $\beta>-1$ and $\gamma>-1$, there exist positive constants $c$ and $C$ such that for all $x\in (-1, 1)$, for all $n\ge 1$, we have
	\begin{equation}\label{key}
	(1- x^{2})^{c}|p_n^{(\beta, \gamma)}(x)|\le C.
	\end{equation}
\end{lemma}

Thus, for all $\beta>-1$ and $\gamma>-1$, the measure $\mu_{\beta, \gamma}$ satisfies all hypotheses of Theorem \ref{t.main1}. Therefore, we obtain the following corollary.
\begin{corollary} Let $\xi_j$ be independent random variables with zero mean and unit variance  and uniformly bounded $(2+\ep)$ moments, for some fixed $\ep>0$.
 Let $N_n$ be the number of real roots of 
$\sum_{j=0}^{n} \xi_j p_j^{(\beta, \gamma)}$, where  $\beta>-1$ and $\gamma>-1$. 	Then given  any compact interval $[a,b]\subset (-1, 1)$, the following holds
\begin{equation} \label{eq:main1:2:jac}
\E N_n([a,b]) = \frac   {\nu([a,b])} {\sqrt 3} n + o(n),
\end{equation}
where $\nu$ is the equilibrium measure of $[-1, 1]$.\footnote{The density of $\nu$ with respect to the Lebesgue measure is $\frac{dx}{\pi\sqrt {1-x^{2}}}$ (ref. \cite[Example 2.1.5]{stahl1992general}, for instance).}

	Assume furthermore that the random variables $\xi_i$ are identically distributed, we have
			\begin{equation} \label{eq:main1:jac}
	\E N_n = \frac 1 {\sqrt 3} n + o(n).
	\end{equation}
	
\end{corollary}

  Beside the Jacobi polynomials, there are many results in the vast literature for orthonormal polynomials that give sufficient conditions for measures $\mu$ that satisfy \eqref{cond:pn}. We refer the interested reader to several resources such as \cite{freudorthogonal, nevai1994orthogonal, stahl1992general}. Below is an example of such conditions when $\K$ is a single interval, which we assume without loss of generality to be $[-1, 1]$.
\begin{theorem}\label{t.main2} Assume that $\mu$ is supported in $\K = [-1, 1]$ with a density $w(x)$ satisfying
\begin{eqnarray}\label{cond:cir}
	\int_{-1}^{1} \frac{dx}{w(x)\sqrt{1-x^{2}}}<\infty.
\end{eqnarray}
Assume that for any compact interval $I\subset (-1,1)$  there exists a constant $C_I>0$ for which
\begin{equation} 
C_I^{-1}\le w(x)\le C_I\quad \text{	almost everywhere on $I$,}\nonumber
\end{equation}
 and given any compact interval $J\subset (0,\pi)$ there exists a constant $C>0$ such that for all sufficiently small $h>0$,
\begin{equation}\label{cond:ls}
\int_J  |w(\cos(\theta+h))  - w(\cos(\theta))|^2 d\theta \le C|h|.
\end{equation}

Then $\mu$ satisfies the hypothesis of Theorem \ref{t.main1} and in particular, given  any compact interval $[a,b]\subset (-1,1)$, the following holds
\begin{equation} 
	\lim_{n\to\infty} \frac 1 n \E N_n([a,b]) = \frac 1 {\sqrt 3} \nu([a,b]),
\end{equation}
	where $\nu$ is the equilibrium measure of $\K = [-1, 1]$ $($that is $d\nu= \frac{dx}{\pi\sqrt{1-x^{2}}})$.
	
	Assume furthermore that the random variables $\xi_i$ are identically distributed, we have
	\begin{equation} 
	\lim_{n\to\infty} \frac 1 n \E N_n(\R) = \frac 1 {\sqrt 3}.
	\end{equation}

\end{theorem}

\section{Local universality of random orthogonal polynomials}\label{s.local}

In this section, we discuss our second set of results, where we prove in particular that many local statistics of the distribution of real roots of $F_n=F_{n,\xi}$  are essentially independent of the distribution of its coefficient vector $\xi=(\xi_0,\dots,\xi_n)$. An example of a local statistics that would be of interest is the average number of roots, the average pairs of roots, and more generally the average number of $k$-tuples of roots of $F_n$ that are inside a given short interval. The rough size of the length of the interval is often called the  scale of the local statistics, and as is common in this subject, the smaller the scale the more difficult the proof would be. On the other hand, there is also a natural lower bound on the scale, since we would want to have typically at least one root inside the interval. In the context of random orthonormal polynomials,  the natural smallest scale is $1/n$, as demonstrated by Gaussian theory \cite{lubinsky2016linear}. Our results below, Theorem~\ref{thm:general:complex}  and Theorem~\ref{thm:general:real}, are stated in this optimal scale, and results for some larger scales could also be obtained by concatenating intervals of the sharp scale obtained here.

Two sequences of random variables $(\xi_1, \dots, \xi_n)$ and $(\tilde \xi_1, \dots, \tilde \xi_n)$ are said to satisfy the moment matching condition {\bf C1} if the following hold, for some constants $N_0, \tau >0$  and $0<\ep<1$. 

{\bf Moment Matching Condition C1.}
\begin{enumerate} [(i)]
	
	\item {\it Bounded $(2+\ep)$ moments:} \label{cond-moment} The random variables $\xi_i$ (and similarly $\tilde \xi_i$), $0\le i \le n$,  are independent (real or complex, not necessarily identically distributed) random variables with unit 
	variance\footnote{For a complex random variable $Z$, the variance of $Z$ is defined to be $\E|Z - \E Z|^{2}$.} and bounded $(2+\ep)$ moments, namely $\E\left |\xi_i \right |^{2+\ep} \le \tau$.

	\item {\it Matching moments to second order with finite exceptions:}\label{cond-matching} For any $i\ge N_0$, for all $a, b\in \{0, 1, 2\}$ with $a+b\le 2$, $$\E \Re\left (\xi_i\right )^{a}\Im \left (\xi_i\right )^{b} = \E \Re\left (\tilde \xi_i\right )^{a}\Im \left (\tilde \xi_i\right )^{b}.$$
	
\end{enumerate}

\begin{theorem}[General Complex universality]\label{thm:general:complex} 
Assume the hypothesis of Theorem \ref{t.main1} for the measure $\mu$ and its orthonormal polynomials $(p_j)$. Assume furthermore that the coefficients $\xi_i$ and $\tilde \xi_i$ satisfy Condition {\bf C1}. Let $k$ be any positive integer and $C>0$. Let $I$ be a compact interval in the interior of $\K$. Then there exist positive constants $C'$ and $c$ independent of $n$  such that the following holds.
	
	For any complex numbers $z_1, \dots, z_k$ in $I+ B(0, c/n)$ and any function $G: \mathbb{C}^{k}\to \mathbb{C}$ supported on \newline$\prod_{i=1}^{k} B (z_i, c/n) $ with continuous derivatives up to order $2k+4$ and $||\triangledown^aG||_{\infty} \le C n^{a}$ for all $0\le a\le 2k+4$, we have 
	\begin{eqnarray}
	\left |\E\sum G\left (\zeta_{i_1}, \dots, \zeta_{i_k}\right) -\E\sum G\left (\tilde \zeta_{i_1}, \dots, \tilde \zeta_{i_k}\right) \right |\le C'n^{-c},\label{gcomplexb}
	\end{eqnarray}
	where the first sum runs over all $k$-tuples\footnote{For example, if $k=2$ and $F_n$ only has two roots $\zeta_1$ and $\zeta_2$, then the first sum is $G(\zeta_1, \zeta_1) + G(\zeta_1, \zeta_2)+G(\zeta_2, \zeta_1)+G(\zeta_2, \zeta_2)$.} $(\zeta_{i_1}, \dots, \zeta_{i_k})$ of the roots $\zeta_1, \zeta_2, \dots$ of $F_n$, and 
	the second  sum runs over all $k$-tuples $(\tilde \zeta_{i_1}, \dots, \tilde \zeta_{i_k})$ of the roots $\tilde \zeta_1, \tilde  \zeta_2, \dots$ of $ \tilde F_n$. 
\end{theorem}

\begin{theorem}[General Real universality]\label{thm:general:real} Assume the hypothesis of Theorem \ref{t.main1} for the measure $\mu$ and its orthonormal polynomials $(p_j)$. Assume furthermore that  $\xi_i$ and $\tilde \xi_i$ are real random variables with mean 0 that satisfy Condition {\bf C1}. Let $k$ be any positive integer and $C>0$.  Let $I$ be a compact interval in the interior of $\K$. Then  there exist positive constants $C'$ and $c$ independent of $n$ such that the following holds.
	
	For any real numbers $x_1,\dots, x_k$, all of which are in $I$, and any function $G: \mathbb{R}^{k}\to \R$ supported on $\prod_{i=1}^{k}[x_i-c/n, x_i+c/n] $ with continuous derivatives up to order $2k+4$ and $||\triangledown^aG||_{\infty}\le C n^{a}$ for all $0\le a\le 2k+4$, we have
	\begin{eqnarray} 
	\left |\E\sum G\left (\zeta_{i_1}, \dots, \zeta_{i_k}\right) 
	-\E\sum G\left (\tilde \zeta_{i_1}, \dots, \tilde \zeta_{i_k}\right) \right |\le C' n^{-c},\nonumber
	\end{eqnarray}
	where the first sum runs over all $k$-tuples $(\zeta_{i_1}, \dots, \zeta_{i_k}) \in \R^{k} $ of the roots $\zeta_1, \zeta_2, \dots$ of $F_n$, and the second  sum runs over all $k$-tuples $(\tilde \zeta_{i_1}, \dots, \tilde \zeta_{i_k}) \in \R^{k}$ of the roots $\tilde \zeta_1, \tilde  \zeta_2, \dots$ of $ \tilde F_n$. 
\end{theorem}

Note that for every complex valued random variable $\xi$ with unit variance, there exists a bivariate Gaussian variable $Z$ with the same mean and covariance matrix, and with $(2+\ep)$ moment bounded by $10$ (say). Therefore, by triangle inequality, one can assume in all the proofs that the random variables $\tilde{\xi}_{i}$ are Gaussian.


Using Theorem \ref{thm:general:real}, we will derive an estimate of number of real roots at a local scale. For each $\ep'>0$, let
$$\K_{\ep'} = \{x\in \K: [x - \ep', x+\ep']\subset \K\}.$$
\begin{theorem}\label{thm:local:expectation}
	Under the hypothesis of Theorem \ref{t.main1}, for any $\ep'>0$, there exists a positive constant $c$ such that the following holds for any interval $[a, b]$ of length $O(1/n)$ inside $\K_{\ep'}$:
	\begin{equation*}\label{key}
\left |\E N_n(a, b) - \E \tilde N_n(a, b)\right | \ll n^{-c}
	\end{equation*}
where $\tilde N_n(a, b)$ is the number of real roots in $(a, b)$ of the polynomial $\tilde F_n(x) = \sum_{i=0}^{n} \tilde \xi_i p_i(x)$ with $\tilde \xi_i$ being iid standard Gaussian.
\end{theorem}

As a corollary, we obtain
\begin{corollary}\label{cor:local:expectation}
	Under the hypothesis of Theorem \ref{t.main1}, for any interval $[a, b]$ (that may depend on $n$) inside $\K_{\ep'}$, it holds that
	\begin{equation*}\label{key}
		\left |\frac{\E N_n(a, b)}{n} - 	\frac{\E \tilde N_n(a, b)}{n} \right |\ll \left (\frac{1}{n}+(b-a)\right )\frac{1}{n^{c}}. 
	\end{equation*}
\end{corollary}

Outside of the set $\K_{\ep'}$, we will show that the number of roots is negligible,
\begin{theorem}\label{thm:edge}
Under the hypothesis of Theorem \ref{thm:real}, we have
	\begin{equation*}\label{key}
	\lim_{\ep'\to 0} \lim _{n\to \infty}\frac{\E N_n(\R \setminus \K_{\ep'})}{n} = 0.
	\end{equation*}
\end{theorem}

 We shall show that Theorems ~\ref{t.main1} and \ref{thm:real} immediately follow from Corollary ~\ref{cor:local:expectation} and Theorem~\ref{thm:edge}, and Theorem~\ref{t.lpx}.

\section{Proof ideas and innovations} 
Since the method in Lubinsky--Pritsker--Xie \cite{LPX1} heavily relies on Kac-Rice formula and the Gaussianity of the $\xi_j$, we need to use a completely different method to prove our results. The approach that we take on, is a recently developed approach via local universality. This was originally introduced by Tao--Vu \cite{TVpoly} and later written in a general framework by the last two authors \cite{nguyenvurandomfunction17}. Though using the same framework, it is technically much more complicated to deal with random orthogonal polynomials. The goal of this paper is to show how to tackle such generality. To the best of our knowledge, this is the first result establishing the number of real roots for general coefficient distribution. 

Since this paper had been written before \cite{nguyenvurandomfunction17} was published on AJM, we decided to write it in a self-contained manner; the parts similar to \cite{nguyenvurandomfunction17} also presented in a slightly different way though the underlying idea is basically the same (Sections \ref{sec:proof:complex} and \ref{sec:proof:real}). As a comparison to \cite{nguyenvurandomfunction17}, we introduce a novel idea that together with the original technique in \cite{nguyenvurandomfunction17} produce a short and elegant proof  for the so-called anti-concentration inequality (see Lemma \ref{cond-smallball}) which is the heart of the proof. In Section \ref{sec:tech}, we establish general estimates for orthogonal polynomials, some of which are technically new and would be of independent interests.

\textbf{Organization.} In Section \ref{sec:tech}, we provide some estimates for $F_n$. 
Using these estimates, we will prove Theorems \ref{thm:general:complex} and \ref{thm:general:real} in Sections \ref{sec:proof:complex} and \ref{sec:proof:real} respectively. Theorem \ref{thm:local:expectation} and Corollary \ref{cor:local:expectation} are proven in Section \ref{sec:pf:local}. Then we prove Theorem \ref{thm:edge} in Section \ref{sec:pf:edge}, Theorems \ref{t.main1} and \ref{thm:real} in Section \ref{sec:pf:main}, Theorem \ref{t.main2} in Section \ref{sec:pf:2}.

\textbf{Notations.} We use standard asymptotic notations under the assumption that $n$ tends to infinity.  For two positive  sequences $(a_n)$ and $(b_n)$, we say that $a_n \gg b_n$ or $b_n \ll a_n$ if there exists a constant $C$ such that $b_n\le C a_n$. If $|c_n|\ll a_n$ for some sequence $(c_n)$, we also write $c_n\ll a_n$. 

If $a_n\ll b_n\ll a_n$, we say that $b_n=\Theta(a_n)$.  If $\lim_{n\to \infty} \frac{a_n}{b_n} = 0$, we say that $a_n = o(b_n)$.  If $b_n\ll a_n$, we sometimes employ the notations $b_n = O(a_n)$ and $a_n = \Omega(b_n)$ to make the idea intuitively clearer or the writing less cumbersome; for example, if $A$ is the quantity of interest, we may write $A = A'+ O(B)$ instead of $A - A' \ll B$, and $A = e^{O(B)}$ instead of $\log A\ll B$.

\section{Several technical estimates} \label{sec:tech}
 In this section we collect and prove several technical estimates for $F_n$ under the assumptions of Theorem~\ref{thm:general:complex} and Theorem~\ref{thm:general:real}. These estimates may be of independent interest. 
 
 Throughout this section, we assume that $\mu$ satisfies the hypothesis of Theorem \ref{thm:general:complex}. In particular, Conditions \ref{cond:w} and \ref{cond:pn} hold. 
 We assume that the $\xi_i$ have uniformly bounded $(2+\ep)$ moments.


For two sets $A, B\subset \mathbb{C}$, we define the sum $A+B$ to be the Minkowski sum $A+B :=\{a+b :  a\in A, b\in B\}$.

In the following, let $I$ be a compact subset of the interior of $\mathcal K$.
All of the constants in this section are allowed to depend on $I$.
 
 Since the degree of $F_n$ is $n$, we start with a simple observation.
\begin{lemma}\label{cond-poly} The number of real roots of $F_n$ is always at most $n$.
\end{lemma}

 \subsection{Upper bound for the $p_j$ and their derivatives}
Next, we observe that Condition \eqref{cond:pn} can be extended to the complex plane (near $I$) and also leads to bounds for the derivatives of $p_n$.
\begin{lemma}
	 Let $L$ be a positive constant. For all $x\in I$, and for all $z\in B(x, L/n)$, we have for all $0\le j\le n$,
	 \begin{eqnarray}\label{eq:bound:pj}
	 |p_j(z)| &=& O(1),\\
	 \label{eq:bound:pj:1dev}
	 |p_j'(z)| &=& O(n),\\
	 \label{eq:bound:pj:2dev}
	 |p_j''(z)| &=& O(n^{2}).
	 \end{eqnarray}
\end{lemma}
\begin{proof}

		To pass to the complex plane, we use the following Bernstein's growth lemma.
		\begin{lemma} [Bernstein's growth lemma] \cite[Theorem 2.2, page 101]{devore1993constructive}
			Let $P$ be a polynomial of degree at most $n$, then for all $z\in \C\setminus [-1, 1]$, 
			\begin{equation*}\label{key}
				|P(z)|\le \left |z+\sqrt{z^{2}-1}\right |^{n} \max \{|P(x)|: x\in [-1, 1]\}.\nonumber
			\end{equation*}
			As a corollary, for any interval $[a, b]$, $L>0$ and $\delta \in (0, 1)$, there exists a constant $C$ depending only on $a, b, L$ and $\delta$ such that for all $z$ with $\Re(z)\in [a+\delta, b-\delta]$ and $|\Im(z)|\le L/n$, we have
			\begin{equation}\label{bernstein}
			|P(z)|\le C \max \{|P(x)|: x\in [a, b]\}.
			\end{equation}
		\end{lemma} 
		Since $I$ is compact, there exists $\delta>0$ such that $I+(-\delta, \delta)$ is a subset of the interior of $\K$. Applying \eqref{bernstein} to the polynomials $p_j$ and the interval $[a, b]: = I+(-\delta, \delta)$, we obtain \eqref{eq:bound:pj}.

		By Cauchy's integral formula, we have
		\begin{eqnarray}
		|p_j'(z)| \quad =\quad  \left |\frac{1}{2\pi i}\int _{\partial B(z, L/n)} \frac{p_j(w)}{(w-z)^{2}}dw\right |  \quad \ll \quad \int _{\partial B(z, L/n)} \frac{1}{n^{-2}}dw \quad \ll \quad n,\nonumber
		\end{eqnarray}
		giving \eqref{eq:bound:pj:1dev}. 
		Similarly for \eqref{eq:bound:pj:2dev}. 
	\end{proof}
We provide an upper bound for $F_n$ in a local ball.
\begin{lemma}  [Boundedness] \label{cond-bddn} For any positive constants $A, L$ and $c_1$,  the following holds: For any $z\in I+B_{\C}(0, L/n)$, with probability at least $1 - O(n^{-A})$, $|F_n(w)|\le \exp(n^{c_1})$ for all $w\in B (z, 2/n)$.
\end{lemma}
\begin{proof}

By \eqref{eq:bound:pj}, 
$$p_j(z) = O(1)$$
uniformly for $w\in B(z, 2/n)$.

By Markov's inequality using the boundedness of the $(2+\ep)$ moments of $\xi_i$, we have 
$$\P\left (|\xi_i|\ge \exp(n^{c_1/2})\right ) \ll  \exp(-n^{c_1/2})$$
So by the union bound, with probability at most $n^{-A}$, we have $$|\xi_i|\le \exp(n^{c_1/2})$$ for all $0\le i\le n$. Under this event, 
\begin{equation*} 
	|F_n(w)|  \ll     \exp(n^{c_1})
\end{equation*}
for all $w\in B (z, 2/n)$. This completes the proof. \end{proof}

\subsection{Lower bound for $Var (F_n)$}
We next show that Condition \eqref{cond:w} implies a lower bound for the variance of $F_n$ on the real line.

\begin{lemma}\label{lm:kernel}
There exists a constant $C'$ such that for all $x\in I$ and for all $n$,
	\begin{equation}\label{cond:kernel:2}
		\sum_{i=0}^{n} |p_i(x)|^{2} \ge  n/C'.
	\end{equation}	 
\end{lemma}
To prove this lemma, we use the following theorem. 
\begin{theorem}\cite[Lemma 3.1 page 100 with Display (4.7) page 25]{freudorthogonal}
	If $\supp \mu\subset [-1, 1]$ with a density $w$ and if for some $x\in (-1, 1)$, it holds that
	\begin{equation}\label{eq:thm:kernel}
	\frac{\int_{t}^{x}w(s)ds}{x-t}\le C<\infty,\quad\forall t\in [-1, 1]
	\end{equation}
	then 
	\begin{equation}\label{eq:thm:kernel:con}
	\frac{1}	{\sum_{k=0}^{n-1} p_n^{2}(x)}\le \frac{64 C}{n} + \frac{\mu(-1, 1)}{n^{2}}\left (\frac{1}{(1-x)^{2}}+\frac{1}{(1+x)^{2}}\right ).
	\end{equation}
\end{theorem}

\begin{proof} [Proof of Lemma \ref{lm:kernel}]
	By translating and rescaling, we can assume without loss of generality that  Assume that $\K\subset [-1, 1]$ and $I = [c, d]$. Since $I$ lies in the interior of $\K$, there exists $\delta>0$ such that $J = [c-\delta, d+\delta]$ also lies in the interior of $\K$. For every $x\in I$, for every $t\in [a, b]$, we aim to show that \eqref{eq:thm:kernel} holds. To this end, consider two cases:
	
	Case 1: $t\in J$. By applying \eqref{cond:w} to the interval $J$, there exists a constant $C_J>0$ such that
	\begin{equation} 
		w(s)\le C_J\quad \text{	almost everywhere on $J$.}\nonumber
	\end{equation}
	Hence, 
	\begin{equation} 
		\frac{\int_{t}^{x}w(s)ds}{x-t}\le C_J.\nonumber
	\end{equation}
	
	Case 2: $t\notin J$. Then, $|t - x|\ge \delta$. We have 
	\begin{equation} 
		\frac{\int_{t}^{x}w(s)ds}{x-t}\le  \frac{\int_{\K}w(s)ds}{\delta}\le \frac{\mu(\R)}{\delta}.\nonumber
	\end{equation}
	Thus, by taking $C = \max\{C_J, \frac{\mu(\R)}{\delta}\}$, condition \eqref{eq:thm:kernel} holds for all $x\in I$. Thus, \eqref{eq:thm:kernel:con} holds which gives
	$$\frac{1}{\sum_{k=0}^{n-1} p_n^{2}(x)}\le \frac{64 C}{n} + \frac{1}{n^{2}}\left (\frac{1}{(1-x)^{2}}+\frac{1}{(1+x)^{2}}\right )\le \frac{64 C}{n} + \frac{2}{\delta^{2}n^{2}}\le \frac{C'}{n},\quad\forall x\in I$$
	as desired.
\end{proof}


%
%
%

Having proved a lower bound for the variance of $F_n$ on the real line, we now pass this to the complex plane.
\begin{lemma} \label{lm:sum:lb}
	 There exist positive constants $c, C''$ such that for all $z\in I+B_\C (0, c/n)$ and for all $n$, we have
	 \begin{equation}\label{cond:kernel:3}
	 \sum_{i=0}^{n} |p_i(z)|^{2} \ge  n/C''.
	 \end{equation}	 
\end{lemma}
\begin{proof}
	
	Let $f(z) = \sum_{i=0}^{n} p_i^{2}(z)$. By triangle inequality, it suffices to show that $|f(z)|\ge n/C'$. Let $x = \Re(z)$. By Lemma \ref{lm:kernel}, we have 
	\begin{equation*} 
	f(x) \ge  n/C'.
	\end{equation*}	 

By \eqref{eq:bound:pj} and \eqref{eq:bound:pj:1dev}, for all $w\in B(x, 1/n)$ and all $0\le i\le n$, we have $|p_j(w)|\le M$ and $|p_j'(w)| \le Mn$ for some constant $M$. Thus, 
$$|f'(w)|\le  \sum_{i=0}^{n} 2 |p_i(w)||p_i'(w)|\le 2M^{2}n^{2}.$$
So, for a sufficiently small $c$, we have for all $z\in B(x, c/n)$, 
\begin{eqnarray}
|f(z)| = \left |f(x) + \int_{x}^{z} f'(w)dw\right | \ge |f(x)| - \int_{x}^{z}  2M^{2}n^{2} dw \ge n/C' - 2M^{2}cn \ge n/C''.\nonumber
\end{eqnarray}
	This completes the proof.	\end{proof}
%

\subsection{Delocalization}
Combining the previous lemmas, we now show that each function $p_i$ only contributes a small amount to the sum. Thinking about $F_n(z)$ as a linear combination of the $\xi_i$ with coefficients $p_i(z)$, one can relate this estimate with the delocalization for the Central Limit Theorem which roughly speaking says that if $a_i = o(\sqrt{\sum_{i=1}^{n} a_i^{2}})$ then the distribution of $\sum_{i=1}^{n} a_i \xi_i$ is close to that of $\sum_{i=1}^{n} a_i \tilde \xi_i$ where $\tilde \xi_i$ are Gaussian random variables.
\begin{lemma}  [Delocalization] \label{cond-delocal}  Let $c$ be the constant in Lemma \ref{lm:sum:lb}. There is a positive constant $C$ such that for every $z\in I+B (0, c/n)$ and for every $i = 1, \dots, n$ the following holds
	$$ |p_i(z)|   \le Cn^{-1/2} \sqrt{\sum _{j = 1}^{n}|p_j(z)|^{2}}.$$ 
\end{lemma}
\begin{proof}
This lemma follows simply by observing from \eqref{eq:bound:pj} and \eqref{cond:kernel:2} that 
$$|p_i(z)|   \ll 1 \ll  n^{-1/2} \sqrt{\sum _{j = 1}^{n}|p_j(z)|^{2}}.$$
\end{proof}

The following lemma bounds the growth of the derivatives of $F_n$.
\begin{lemma}  [Derivative growth]\label{cond-repulsion} Let $c$ be the constant in Lemma \ref{lm:sum:lb}. There is a positive constant $C$ such that for any real number $x\in I + B(0, c/n)$, we have
	\begin{equation}
	\sum_{j=1}^{n} |p_j'(x)|^{2}\le Cn^{2}\sum_{j=1}^{n} |p_j(x)|^{2},\nonumber
	\end{equation}
	and
	\begin{equation}
	\sum_{j=1}^{n} \sup_{z\in B(x, 1/n)}|p_j''(z)|^{2}\le Cn^{4}\sum_{j=1}^{n} |p_j(x)|^{2}.\nonumber
	\end{equation}
\end{lemma}
\begin{proof}Let $c_1 = 1$. From \eqref{eq:bound:pj:1dev} and \eqref{cond:kernel:2}, we get
\begin{equation} 
\sum_{j=1}^{n} |p_j'(x)|^{2}\ \  \ll \ \  n^{3} \ \ \ll \ \ n ^{2}\sum_{j=1}^{n} |p_j(x)|^{2}.\nonumber 
\end{equation}
Similarly, by \eqref{eq:bound:pj:2dev} and \eqref{cond:kernel:2}, we get
\begin{equation} 
\sum_{j=1}^{n} \sup_{z\in B_{\C}(x, 1/n)}|p_j''(z)|^{2}\ \ \ll \ \ n^{5}\ \ \ll \ \ n ^{4}\sum_{j=1}^{n} |p_j(x)|^{2}.\nonumber
\end{equation}
These estimates complete the proof of Lemma \eqref{cond-repulsion}.
\end{proof}

 \subsection{Anti-concentration}
In the next lemma, we show that the minimum of $F_n$ in a local ball is bounded away from 0 with high probability. This is the key lemma of our proof. 
\begin{lemma} [Anti-concentration]\label{cond-smallball}  For any positive constants $A, c$ and $c_1$,  the following holds: For every $z\in I+B(0, c/n)$, with probability at least $1 - O(n^{-A})$, there exists $z'\in B (z, c/n)$ for which $|F_n(z')|\ge \exp(-n^{c_1})$.
\end{lemma}
 In fact, we shall prove the following more general result, which applies to a more general type of random orthonormal polynomials
$$F_{(b_{j, n}), n}(x)=\sum_{j=1}^n b_{j, n} \xi_j p_j(x)$$
where $b_{j, n}$ are deterministic numbers that may depend on $n$.

\begin{lemma} \label{lm:anti:general}  Let $A, c, c_1,  \ep$ be positive constants. Assume that the random variables $\xi_{i}$ independent with uniformly bounded $(2+\ep)$-central moments and 
$$e^{- i^{c_1/2}} \ll |b_{i, n}| \ll e^{  i^{c_1/4}}$$  
uniformly  for all $i, n\geq 0$.  

Then there exists a constant $C>0$  such that the following holds. For every $x\in I$,  there exists $x'\in[x-c/n, x+c/n]$  such that 
$$\sup_{z\in \R}\P(|F_{(b_{j, n}), n}(x')-Z|\leq e^{-n^{c_1}})\leq Cn^{-A}.$$
\end{lemma}
Note that we do not need to assume $\E \xi_i$ to be uniformly bounded for this lemma. 
\begin{proof} [Proof of Lemma \ref{lm:anti:general}]  For simplicity of notation, we write $F_n$ instead of $F_{(b_{j, n}), n}$, and $b_j$ instead of $b_{j, n}$ in the proof. By replacing $c_1$ by a slightly smaller constant and $A$ by a larger constant, it suffices to show that there exists some $x'\in [x-c/n, x+c/n]$ such that
\begin{equation}\label{eq:smallball:F}
 \sup_{Z\in \R} \P(|F_{n}(x')-Z|\leq e^{-n^{c_1}}n^{-A/2})\leq Cn^{-A/2}.
\end{equation}
We first prove this inequality for the case when $\xi_{i}$ are iid Rademacher by using the following Halasz type inequality.

\begin{lemma} (\cite{halasz1977} or \cite[Lemma 9.3]{nguyenvurandomfunction17}) \label{lm:halasz} Let $\ep_{1}$, \dots , $\ep_{N}$  be independent Rademacher random variables. Let $a_{1}$, \dots, $a_{N}$  be real numbers and $l$  be a fixed integer. Assume that for any distinct sets $\{i_{1}, \dots, i_{l'}\}$  and $\{j_{1}, \dots, j_{l''}\}$  where $l'+l''\leq 2l$,  one has $|a_{i_{1}}+\cdots+a_{i_{l'}} -a_{j_{1}}-\cdots-a_{j_{l''}}|\geq a$  then for any number $Z,$
$$
\P\left (\left |\sum_{j=1}^{N}a_{j}\ep_{j}-Z\right |\leq aN^{-l}\right )=O_{l}(N^{-l}).
$$
\end{lemma}
Let $1>\alpha>0$ and $l\in \mathbb{N}$ be constants to be chosen. Conditioning on all the $\xi_{i}$ with $i\geq n^{\alpha}$, it suffices to show the following version of \eqref{eq:smallball:F}:
\begin{equation}\label{eq:smallball:F:1}
\sup_{Z\in \R}\P\left (\left |\sum_{j=0}^{n^{\alpha}} b_j \xi_j p_j(x') - Z\right |\leq e^{-n^{c_1}}n^{-A/2}\right )\leq Cn^{-A/2}.
\end{equation} 

For every distinct subsets $\{i_{1}, \dots, i_{l'}\}$ and $\{j_{1}, \dots, j_{l''}\}$ of the set $\{0, 1 , \dots, n^{\alpha}\}$ with $ l'+l''\leq 2l$, let 
\begin{eqnarray*}
E&=&E_{i_{1},\ldots,i_{l'},j_{1},\ldots,j_{l''}}\\
&=&\left \{x'\in[x-c/n, x+c/n]: \left |  \sum_{s=1}^{l'}b_{i_{j}}p_{i_{s}}(x')-\sum_{s=1}^{l''}b_{j_{s}}p_{j_{s}}(x')\right |\leq e^{-n^{c_1}}\right \}
\end{eqnarray*}

To use Lemma \ref{lm:halasz}, we aim to find an $x'$ that lies outside of all of these sets. Our idea is to use the classical Remez inequality to control their size.

\begin{lemma}(Remez inequality).  Let $P$  be a polynomial of degree $d$  in $\mathbb{R}$.  Then for any interval $J\subset \mathbb{R}$ and every measurable subset $E\subset J$,  we have
$$
\max_{J}|P|\leq \left (\frac{4|J|}{|E|}\right )^{d}\sup_{E}|P|.
$$
\end{lemma} 
Let $J$ be the smallest interval that contains $\K+[-1/n, 1/n]$. Applying Remez inequality to $P=  \sum_{s=1}^{l'}b_{i_{s}}p_{i_{s}}(x')-   \sum_{s=1}^{l''}b_{j_{s}}p_{j_{s}}(x')$ , one gets
$$
M :=\max_{J}|P|\leq \left (\frac{4|J|}{|E|}\right )^{n^{\alpha}}e^{-n^{c_1}}.
$$

Using orthogonality of $p_{i},$
\begin{equation}\label{eq:ortho:M2}
  M^{2}  \ \ \geq \ \   \frac{1}{\mu(J)}\int P^{2}d\mu \ \ =\frac{1}{\mu(J)} \left (\ \  \sum_{s=1}^{l'}|b_{i_{s}}|^{2}+\sum_{s=1}^{l''}|b_{j_{s}}|^{2}\right )  \ \ \gg \ \ e^{-n^{c_1/2}}.
\end{equation}
Thus,
$$
e^{n^{c_1/2-\alpha}}\leq\frac{C}{|E|}.
$$
Choosing $\alpha=c_1/4$, we have $|E_{i_{1},\ldots,i_{l'},j_{1},\ldots,j_{l''}}|=O(e^{-n^{c_1/4}})$. By the union bound, the Lebesgue measure of the union of all such sets $E_{i_{1},\ldots,i_{l'},j_{1},\ldots,j_{l''}}$ is at most $n^{2\alpha l}O(e^{-n^{c_1/4}})=O(e^{-n^{c_1/8}})$. Thus, there exists an $x'\in[x-c/n, x+c/n]$ that lies outside of all $E_{i_{1},\ldots,i_{l'},j_{1},\ldots,j_{l''}}$. Applying Lemma \ref{lm:halasz} for these $x'$ with $a = e^{-n^{c_1}}$ and $N = n^{\alpha}$ yields
\begin{equation} 
\sup_{Z\in \R}\P\left (\left |\sum_{j=0}^{n^{\alpha}} b_j \xi_j p_j(x') - Z\right |\leq e^{-n^{c_1}}n^{-\alpha l}\right )\leq Cn^{-\alpha l}.
\end{equation} 
Setting $ l=A/2\alpha$, we obtain \eqref{eq:smallball:F:1} and thus \eqref{eq:smallball:F} for the Rademacher case.

Next, assume that $\xi_{i}$ are symmetric for all $i$, i.e., $\xi_{i}$ and $-\xi_{i}$ have the same distribution. Since $(\xi_{i})_{i}$ and $(\xi_{i}\ep_{i})_{i}$ have the same distribution where $\ep_{1}$, \dots, $\ep_{n}$ are independent Rademacher random variables, \eqref{eq:smallball:F} is equivalent to proving that for all $Z$,
$$
\P\left (\left |\sum_{j=0}^{n}b_{j}\xi_{j}\ep_{j}p_{j}(x')-Z\right |\leq e^{-n^{c_1}}n^{-A/2}\right )\leq Cn^{-A/2}
$$
for some $x'\in[x-c/n, x+c/n]$. The latter can be deduced from
$$
\int_{x-c/n}^{x+c/n}\P\left (\left |\sum_{j=0}^{n}b_{j}\xi_{j}\ep_{j}p_{j}(x')-Z\right |\leq e^{-n^{c_1}}n^{-A/2}\right )dx'\leq Cn^{-A/2-1}
$$
which, by Fubini's theorem, is the same as
\begin{equation}\label{eq:e:smallball}
  \E _{\xi_{1},\ldots,\xi_{n}}\int_{x-c/n}^{x+c/n}\P_{\ep_1, \dots, \ep_{n}} \left (\left |\sum_{j=0}^{n}b_{j}\xi_{j}\ep_{j}p_{j}(x')-Z\right |\leq e^{-n^{c_1}}n^{-A/2}\right )dx'\leq Cn^{-A/2-1}
\end{equation}
where the expectation is with respect to the randomness of the random variables $\xi_1, \dots, \xi_n$ and the probability is with respect to the randomness of $\ep_1, \dots, \ep_n$.

To prove the last inequality, notice that from the boundedness of the $ (2+\ep)$-moments of $\xi_{i}$ and symmetry of $\xi_{i}$, there exist positive constants $d$ and $q$ such that $\P(|\xi_{i}|<d)\leq q<1$. Indeed, if for some $d>0, \P(|\xi_{i}|<d)>1-d$, then by H\H{o}lder's inequality
\begin{eqnarray*}
1 &\leq& \E |\xi_{i}|^{2}=\E |\xi_{i}|^{2}1_{|\xi_{i}|<d}+\E |\xi_{i}|^{2}1_{|\xi_{i}|\geq d}\\
&\leq& d^{2}+d^{\ep/(2+\ep)}(\E |\xi_{i}|^{2+\ep})^{2/(2+\ep)}.
\end{eqnarray*}

Thus, one  chooses $d$ small enough and $q=1-d$ to have $\P(|\xi_{i}|<d)\leq q<1.$

Now, by Chernoff's inequality, with probability at least $1-e^{-\Theta(n^{\alpha})}$, there are at least $(1-q)n^{\alpha}/2$ indices $j\in\{0, \dots, n^{\alpha}\}$ for which $|\xi_{j}|\geq d$. On the event that this happens, we show that
\begin{equation}\label{eq:e:smallball:1}
\int_{x-c/n}^{x+c/n}\P_{\ep_1, \dots, \ep_{n}} \left (\left |\sum_{j=0}^{n}b_{j}\xi_{j}\ep_{j}p_{j}(x')-Z\right |\leq e^{-n^{c_1}}n^{-A/2}\right )dx'\leq Cn^{-A/2-1}.
\end{equation}
Assuming this, the left-hand side of \eqref{eq:e:smallball} is at most $Cn^{-A/2-1}+ e^{-\Theta(n^{\alpha})}\ll n^{-A/2-1}$, proving \eqref{eq:e:smallball}.

To prove \eqref{eq:e:smallball:1}, we condition on all $\ep_{j}$ with $j>n^{\alpha}$ and all $\ep_j$ with $|\xi_{j}|<d$. So, the remaining randomness comes from at least $(1-q)n^{\alpha}/2$ random variables $\ep_j$. We obtain from the above argument for the Rademacher case that for all $x'$ outside a subset of $[x-c/n, x+c/n]$ of measure at most $Ce^{-n^{c_1/8}}$, we have
\begin{eqnarray*}
\P_{\ep_1, \dots, \ep_n} \left (\left |\sum_{j=0}^{n}b_{j}\xi_{j}\ep_{j}p_{j}(x')-Z\right |\leq e^{-n^{c_1}}n^{-A/2}\right )  &\leq& Cn^{-A/2}.
\end{eqnarray*}
Integrating with respect to $x'$ gives \eqref{eq:e:smallball:1} and therefore \eqref{eq:smallball:F} for the symmetric case.

For the general case, assume that \eqref{eq:smallball:F} fails. That means there exists $Z\in \R$ such that
$$\P(|F_{n}(x')-Z|\leq e^{-n^{c_1}}n^{-A/2})> Cn^{-A/2}.$$
Let $\xi_{1}'$, \dots, $\xi_{n}'$ be independent copies of $\xi_{1}$, \dots, $\xi_{n},$ respectively. Let $\xi_{i}''=\xi_{i}-\xi_{i}'$ then $\xi_{i}'$ is symmetric and has uniformly bounded $( 2+\ep)$ moment. Then
\begin{eqnarray*}
&& \P\left (\left |\sum_{j=0}^{n}b_{j}\xi_{j}''p_{j}(x')\right |\leq 2e^{-n^{c_1}}n^{-A/2}\right ) \\
&\geq& \P\left (\left |\sum_{j=0}^{n}b_{j}\xi_{j}p_{j}(x')-Z|\leq e^{-n^{c_1}}n^{-A/2}, |\sum_{j=0}^{n}b_{j}\xi_{j}p_{j}(x')-Z\right |\leq e^{-n^{c_1}}n^{-A/2}\right )\nonumber\\
&\geq& Cn^{-A}.\nonumber
\end{eqnarray*}
 
This contradicts the symmetric case where $A$ is replaced by $2A$, and hence completes the proof. 
\end{proof}

%
%
%
%
%

\section{Proof of local universality for complex roots, Theorem \ref{thm:general:complex} }\label{sec:proof:complex}
In this section and the next, we prove the universality Theorems \ref{thm:general:complex} and \ref{thm:general:real}. They are similar in ideas to \cite{nguyenvurandomfunction17} by the last two authors. We nevertheless present them with two goals. Firstly, to make the paper self-contained and allow the reader not familiar with \cite{nguyenvurandomfunction17} to get a full picture of the proof. Secondly, we present them in a slightly different way so those who have read \cite{nguyenvurandomfunction17} can have a refreshed view.

Before going to the proof, let us outline a sketch.
\begin{itemize}
	\item Step 1: First we reduce to the basic case that $k=1$ and $G$ is a function of one complex variable. This case already contains all of the key ideas.
	\item Step 2: Apply Green's formula to reduce the local statistics to an explicit expression for $F_n$
	$$\E \sum_{j} G_1(\zeta_j) =\E \int_{\mathbb C}\log |F_n(z)| \times \text{(a deterministic function) } dz.$$
	We then replace the $\log |F_n(z)|$ (which blows up when $F_n(z)$ is too big or too small) by a smooth function $K_n(\log |F_n(z)|)$ where $K_n$ is a nice, bounded function 
	\begin{eqnarray}
	\E \int_{\mathbb C}\log |F_n(z)| \times \text{(a deterministic function) } dz &\approx& \E \int_{\mathbb C} K_n(\log |F_n(z)|)dz \notag\\
	&=& \int_{\mathbb C} \E K_n(\log |F_n(z)|) dz.\notag
	\end{eqnarray}
	\item Step 3: Use Lindeberg swapping technique to perform the comparison with the Gaussian counterpart
	$$ \E K_n(\log |F_n(z)|)\approx \E K_n(\log |\tilde F_n(z)|).$$ 
	
\end{itemize}
We now go into the details. 
\subsection{Step 1: Reduction}
We first describe some mild simplifications in the proof presented in this section. First, without loss of generality we may assume that $\widetilde \xi_j$ are all Gaussian. For the sake of simplificity and  brevity, we will assume that for any $j=0,...,n$  the coefficients $\xi_j$ and $\widetilde \xi_j$   have matching moments up to order $2$.  We will also assume that the test functions $G$ have the tensor form $G(\eta_1,\dots, \eta_k)=G_1(\eta_1)\dots G_1(\eta_k)$ where, for each $j$, $G_j$ is supported in $B(z_j,\frac c {n})$ with sufficient smoothness and so that $\|G_j^{(a)}\|_{\infty} = O(n^a)$ for $0\le a \le 2$. Here $c>0$ is a sufficiently small constant that does not depend on $n$. By a standard argument involving multiple Fourier series expansions,  the general test functions specified in Theorem~\ref{thm:general:complex} could be well approximated  by weighted average of these simpler functions. The amount of regularity required for  $G$ in the proof could be easily tracked, and we deemphasize this aspect of the proof  to improve the clarity and the brevity of the argument. Furthermore, we will only show the details for the case $k=1$, extension to the general case  is fairly straightforward, see e.g. \cite{DOV,do2019, nguyenvurandomfunction17,TVpoly}.

\subsection{Step 2: reduce to a nice integral of $F_n$}
We now transform the local statistics about the roots of $F_n$ to a more convenient integral form. Let $(\zeta_j)$ be the complex zeros of $F_n$ (multiplicity allowed). 
Using Green's formula, we write
$$\E \sum_{j} G_1(\zeta_j) =\E \int_{\mathbb C}\log |F_n(z)| \left (-\frac 1{2\pi}\Delta G_1(z)\right )dz, $$
where $dz$ is the two dimensional Lebesgue measure.
Let $H(z)=-\frac 1 {2\pi} \Delta G_1(z)$, which is bounded above by $O(n^2)$ and supported inside $B(z_1, \frac c{n})$.  
\begin{remark}\label{r.upperN}
Since $\|G\|_\infty =O(1)$, the above display can always be bounded above, up to some constant, by $N(B(z_1,\frac c {n}))$, the number of (complex) zeros of $F_n$ inside $B(z_1,\frac c{n})$.
\end{remark}

We now aim to approximate $\log |F_n(z)|$ by $K_n(\log|F_n(z)|)$ where $K_n$ is a nice function.
Let $\varphi:\mathbb R\to\mathbb R$ be a smooth approximation of the  function $f(x)=x$ near $0$ such that $|\varphi(x)|\le |x|$ everywhere, and furthermore $\varphi$ is supported on $[-2,2]$ and equals $x$ in $[-1,1]$. Let $M_n$ be a large deterministic number that may depend on $n$ and let $K_n(x)=M_n \varphi(x/M_n )$. Note that $|K_n(x)|\le min(|x|,2M_n)$ by construction, and for any $0\le \alpha\le 3$, it holds that 
$$\|K_n^{(\alpha)}\|_{\infty} \le M_n^{-2}.$$

For convenience of notation, let 
$$D_n=I+B_{\mathbb C}(0, c/n)$$
the Minkowski sum of these two sets.  We shall always assume that $z_1\in D_n$ and $c>0$ is a sufficiently small constant in the proof. The next lemma controls the second moment of the $\log|F_n(z)|$ which we shall use immediately after to control the error term of the approximation.
\begin{lemma}\label{l.badevent} For any $A,c>0$ we may find an event $A_0$ such that
$$\P(A_0) = 1- O_A(n^{-A})$$
and
$$\E \Big[1_{A_0} \int_{B(z_1,\frac 1{10n})} |\log |F_n(z)||^2 dz\Big] = O(n^{c}).$$
\end{lemma}
We note that a similar estimate for random algebraic polynomials was proved in \cite[Theorem 11]{do2019}.
\proof
Let $c_1=c/3$. Using Lemma~\ref{cond-bddn} and Lemma~\ref{cond-smallball}, we may let $A_0$ be the event where:\\
\begin{itemize}
\item[(i)] for every $w\in B(z_1,\frac 2 n)$, we have \ \ \ \ $\log |F_n(w)| \le n^{c_1}$.
\item[(ii)] there is $w_0\in  B(z_1, \frac 1 {10n})$ such that  \ \ \ \  $\log |F_n(w_0)|\ge -n^{c_1}$.
\end{itemize}

Now, $B(z_1,\frac 1{10n})\subset   B(w_0,\frac 1 {5n})$, thus it suffices to show that 
$$1_{A_0}\int_{B(w_0,\frac 1{5n})}  |\log |F_n(z)||^2 dz = O(n^{3c_1}).$$

Using the classical Jensen's bound for the number of roots (see for example \cite[Equation 8]{nguyenvurandomfunction17}), it follows from (i) and (ii) that
\begin{eqnarray*}
N\left (B(w_0,\frac 1{n})\right ) 
&\le& \frac{1}{\log (3/2)}\left (\sup_{|z-z_1|=\frac 3{2n}} \log |F_n(z)| - \log |F_n(w_0)|\right ) \\
&<&   5n^{c_1}.
\end{eqnarray*}
It follows that there is some $\frac 1 {5n} \le r \le \frac 1{n}$ such that $F_n$ has no zeros in the annulus $\{r- \frac 1 3 n^{-(1+c_1)}  \le |z-w_0|\le r+ \frac 1 3 n^{-(1+c_1)}\}$. Now, for any $w\in B(w_0,\frac 1 n)$ we have
$$\int_{B(w_0,1)}|\log |z-w||^2 dz \le  \int_{B(0,2)}|\log |z||^2 dz \ll 1.$$
therefore it suffices to show that
\begin{equation}\label{e.logf}
\int_{B(w_0,\frac 1{5n})}  |\log |f||^2 dz = O(n^{3c_1})
\end{equation}
where $f(z)=\frac {F_n(z)}{\prod_j (z-w_j)}$ and $w_1,\dots, w_m$ are the zeros of $F_n$  inside $B(w_0,r-\frac{1}{4}n^{-(1+c_1)})$. It is clear that $f$ is analytic on $B(w_0,r+\frac{1}{4}n^{-(1+c_1)})$, and using the maximum principle we have
$$\sup_{|z-w_0|\le r+\frac{1}4 n^{-(1+c_1)}}  \log |f(z)| \le \sup_{|z-w_0| = r+\frac{1}4 n^{-(1+c_1)}}  \log |F_n(z)|+ Cn^{c_1}\log n \le Cn^{2c_1},$$
for some absolute constant $C$. At the same time we also have 
$$\log |f(w_0)|\ge \log |F_n(w_0)|\ge -n^{-c_1}.$$
It follows from an application of Harnack's inequality that
$$\log |f(z)| \ge -Cn^{3c_1}$$
uniformly over $z\in B(w_0,r)$ (which contains $B(w_0,\frac 1 {5n})$), for some absolute constant $C$. The desired estimate \eqref{e.logf} follows immediately. This completes the proof of Lemma~\ref{l.badevent}.
\endproof

Let $A_0$ be the event given by Lemma~\ref{l.badevent}. Let $N=N(B(z_1,\frac c {n}))$ below. Using Lemma~\ref{cond-poly} and Remark~\ref{r.upperN} we have  
\begin{eqnarray*}
 && \E \Big(|\int \log |F_n(z)|  H(z) dz| 1_{A_0^c}\Big)   \quad + \quad  \E \Big( |\int K_n(\log |F_n(z)|)  H(z) dz| 1_{A_0^c} \Big) \\
 &\ll& n \P(A_0^c) + \|K_n\|_{\infty} \P (A_0^c).
\end{eqnarray*}
On the other hand, using $|K_n(x)|\le |x|$ and the fact that $K_n(x)=x$ for $|x|\le M_n$ we have
\begin{eqnarray*}
&&   \E \left (1_{A_0}\left [\int_{\mathbb C}\log |F_n(z)|  H(z)dz  -   \int _{\C} K_n(\log|F_n(z)|)  H(z)dz\right ]\right )\\
 &\ll& \E\Big(1_{A_0} \int_{B(z_1,c/n)} |\log |F_n(z)||  1_{\{z: |\log |F_n(z)||\ge M_n\}} dz\Big)\\
 &\ll& M_n^{-1} \E\Big(1_{A_0} \int_{B(z_1, c/n)} |\log |F_n(z)||^2   dz\Big)\\
 &\ll& M_n^{-1} n^{c}, \quad \text{thanks to Lemma~\ref{l.badevent}}.
\end{eqnarray*}
Collecting estimates, it follows that
\begin{eqnarray*}
&&   \E\Big[\int_{\mathbb C}\log |F_n(z)|  H(z)dz  -   \int  K_n(\log|F_n(z)|)  H(z)dz\Big] \\
&\ll& n^{1-A}  + M_n n^{-A} + M_n^{-1} n^{c}.
\end{eqnarray*}

\subsection{Step 3: Comparison with the Gaussian counterpart}
Now, we will prove the following estimate, where we recall that $\widetilde F_n$ is the Gaussian analogue of $F_n$.
\begin{lemma}\label{l.lindeberg} The following holds uniformly over  $z\in B(z_1,\frac c {n})$ and $K_n:\R\to \C$:
\begin{equation}\label{e.lindeberg}
|\E K_n(\log|F_n(z)|)-\E K_n (\log |\widetilde F_{n}(z)|)| \quad \ll \quad n^{-\alpha_1 \ep/4}\max_{0\le a\le 3} \|K_n^{(a)}\|_{\infty}.
\end{equation}
\end{lemma}
\proof Fix $z$. Note that $V_n(z):= Var[F_n(z)]$  does not depend on the distribution of the coefficients. We shall let $k_n(x)=K_n(x+\frac 1 2 \log V_n(z))$, then it is clear that
$$\|k_n^{(a)}\|_{\infty} =\|K_n^{(a)}\|_{\infty}$$ 
for all  $a\ge 0$.

We now approximate $k_n(\log |.|)$ with a nicer function $H_n(.)$. For some $M>0$ to be chosen later (that may depend on $n$), we consider a smooth step function $\psi:\R\to [0,1]$ such that $\psi=0$ on $(-\infty, -M-1]$ and $\psi=1$ on $[-M,\infty)$. Such a function can be easily constructed such that
$$\|\psi^{(a)}\|_{\infty}=O_a(1), \quad a\ge 0.$$
We note that $\psi^{(a)}$ is supported on $[-M-1, -M]$ if $a\ge 1$. We then define
$$H_n(w)= \psi(\log |w|) k_n(\log |w|), \quad w\in \C.$$
Viewing $H_n$ as a function on $\R^2$ (instead of $\C$), for any 2-index $\alpha$ with $|\alpha|\le 3$ we have
\begin{eqnarray*}
\|H_n^{(\alpha)}\|_{\infty} &\ll& e^{3M} \max_{0\le a\le 3} \| k_n^{(a)}\|_{\infty}
\end{eqnarray*}

For each $0\le j \le n $, let 
$$\varphi_j(z)=\sum_{k< j} \widetilde \xi_{k} p_k(z)+ \sum_{k > j} \xi_k p_k(z).$$

We then estimate
\begin{eqnarray*}
&&  \left |\mathbb E  H_n \left (\frac {F_n(z)}{\sqrt{V_n(z)}}\right ) - \E H_n \left (\frac{\widetilde F_{n}(z)}{\sqrt{V_n(z)}}\right )\right  |\\
&\le& \sum_{j=0}^n  \left |\mathbb E H_n \left (\frac{\varphi_j(z) + \xi_j p_j(z)}{\sqrt{V_n(z)}}\right ) - \mathbb E  H_n\left  ( \frac{\varphi_j(z) +  \widetilde \xi_j  p_j(z)}{\sqrt{V_n(z)}}\right ) \right |.
\end{eqnarray*}

For each $1\le j\le n$, let $H_{j,n}(w)= H_n\left (\frac {\varphi_j(z)}{\sqrt{V_n(z)}}+w\right )$ and view this as a function on $\R^2$. 

For each multiindex $\alpha$, we write $H^{(\alpha)}_{j,n}$ for the corresponding mixed partial derivative. 

Let $\Delta_{j,n}$ be the error term in the Taylor expansion around $(0,0)$ for $H_{j,n}$ up to second order.  Using the mean value theorem, it is clear that the following holds for both $s=2$ and $s=3$:
\begin{eqnarray*}
|\Delta_{j,n}| &\ll& \left (\frac{|p_j(z)|}{\sqrt{V_n(z)}}\right )^s  (|\xi_j|^s +|\widetilde \xi_j|^s)\sum_{|\alpha|=s}\|H_{j,n}^{(\alpha)}\|_{\infty}.
\end{eqnarray*}

Using the fact that $\xi_j$ and $\widetilde \xi_j$ have matching moments up to order $2$,  for any $t\in [2,3]$ we have
\begin{eqnarray*}
\left  | \mathbb E  H_{j,n} \left (\frac{\xi_j p_j(z)}{\sqrt{V_n(z)}}\right ) -\mathbb E  H_{j,n} \left ( \frac{\widetilde \xi_j  p_j(z)}{\sqrt{V_n(z)}}\right ) \right |  
&\ll& \left (\frac{|p_j(z)|}{\sqrt{V_n(z)}}\right )^t \E (|\xi_j|^t +|\widetilde \xi_j|^t)\sum_{2\le |\alpha|\le 3}\|H_{j,n}^{(\alpha)}\|_{\infty}. 
\end{eqnarray*}
Take $t=2+\ep$ where $\ep>0$ such that $\E |\xi_j|^{2+\ep}=O(1)$. Using Lemma~\ref{cond-delocal} and properties of $H_{j,n}$ we can bound the last display by
$$\ll n^{-\alpha_1 \ep} \left (\frac{|p_j(z)|}{\sqrt{V_n(z)}}\right )^{2} \max_{|\alpha|\le 3} \|H^{(\alpha)}_n\|;$$
therefore by summing over $j$, we obtain
\begin{eqnarray*}
  \left |\mathbb E  H_n \left (\frac {F_n(z)}{\sqrt{V_n(z)}}\right ) - \E H_n \left (\frac{\widetilde F_{n}(z)}{\sqrt{V_n(z)}}\right ) \right | 
&\ll&   n^{-\alpha_1 \ep}  \max_{|\alpha|\le 3} \|H^{(\alpha)}_n\|.
\end{eqnarray*}

Let $h_n(w):=1-\psi(\log |w|)$ which is a nonnegative function. Then
$$\Big|\E \Big [(k_n -\psi k_n)(\log|\frac{F_n(z)}{\sqrt {V_n(z)}}|) \Big] \Big| \le \|K_n\|_{\infty} \E \Big(h_n(\frac{F_n(z)}{\sqrt {V_n(z)}})\Big).$$
We now apply the same (swapping) argument as before for $h_n$ instead of $H_n$.  It follows that
\begin{eqnarray*}
\E \left (h_n\left (\frac{F_n(z)}{\sqrt {V_n(z)}}\right ) \right )
&=& \E \left (h_n\left (\frac{\widetilde F_n(z)}{\sqrt {V_n(z)}}\right ) \right ) + O\left (n^{-\alpha_1\ep} \max_{|\alpha|\le 3} \|^{(\alpha)}_n\|_{\infty}\right )\\
 &\ll& \P\left (\frac{|\widetilde F_n(z)|}{\sqrt {V_n(z)}}\le e^{-M}\right ) + O(n^{-\alpha_1\ep} e^{3M})\\
 &\ll&  e^{-M} + n^{-\alpha_1\ep} e^{3M}.
\end{eqnarray*}

Collecting estimates, it follows that
\begin{eqnarray*}
  |\mathbb E  K_n (F_n(z))  - \E K_n (\widetilde F_{n}(z)) | 
&\ll&  (n^{-\alpha_1 \ep}  e^{3M} + e^{-M}) \max_{0\le a\le 3} \|K^{(a)}_n\|.    
\end{eqnarray*}
By taking $M=\log(n^{\alpha_1\ep/4})$, we obtain \eqref{e.lindeberg}. This completes the proof of Lemma~\ref{l.lindeberg} \endproof

Using Lemma~\ref{l.lindeberg} and the triangle inequality, it follows that
\begin{eqnarray*} 
 \left  |\E \int  K_n(\log|F_n(z)|)  H(z)dz -  \E \int  K_n(\log|\widetilde F_{n}(z)|)  H(z)dz\right |  
&\ll& n^{-\alpha_1 \ep/4} M_n^{2}.
\end{eqnarray*}

Collecting estimates, we obtain
\begin{eqnarray*}
&& |\E \int_{\mathbb C}\log |F_n(z)|  H(z)dz - \E \int_{\mathbb C}\log |\widetilde F_{n}(z)|  H(z)dz|\\
&\ll& n^{1-A}  + M_n n^{-A} + M_n^{-1} n^{c} +  n^{-\alpha_1 \ep/4} M_n^{2}  \\
&\ll& n^{-c'}
\end{eqnarray*}
by choosing $M_n=n^{c'}$ with $c'>0$ small and $A>0$ sufficiently large. \endproof

\section{Proof of local universality for real roots, Theorem~\ref{thm:general:real}} \label{sec:proof:real}
Before going into the proof, let us again start with a road map.
\begin{itemize}
	\item Step 1: reduce to the basic case $k=1$ and $G$ is a smooth function of one real variable.
	\item Step 2: in order to reduce to the case that $G$ is a smooth function of one complex variable, we try to extend $G$ continuously to the complex plane and show that one does not lose much during the process. In other words, we need to show that the number of complex roots very close the real line is negligible. This can be explained by the so-called repulsion property of the roots: they tend to repel each other. So, if there is at least one such complex root $\zeta$, its complex conjugate $\bar \zeta$, which is another root, stays very close to $\zeta$; and that is quite unlikely. So, in this step, we provide an upperbound for the probability of having two close roots.
	
	\item Step 3: to finish, we implement the extension of $G$ to the complex plane as mentioned above and use Theorem \ref{thm:general:real}.
\end{itemize}

\subsection{Step 1: Reduction}
We first describe some mild simplifications in our proof. We first note that since $\xi_j$ are real the complex roots of $F_n$ are symmetric with respect to the real line. Without loss of generality we may assume that $\widetilde \xi_j$ are Gaussian. For simplicity, we will also assume that for all $j=1,..,n$ the coefficients $\xi_j$ and $\widetilde \xi_j$ satisfy the matching moment Condition {\bf C1}. In other words, we assume that $N_0$ in Condition {\bf C1} equals 1. The proof for general constants $N_0$ is similar.  Furthermore, we will only consider test functions of tensor form $G(\eta_1,...., \eta_k)=G_1(\eta_1)...G_k(\eta_k)$. We will also only show the details for the case $k=1$, the general case is entirely similar for these test functions. 

\subsection{Step 2: Repulsion} The key idea in the proof is the following observation: with high probability the roots of $F_n$ in a thin strip around local real intervals  are all real. Thanks to this observation, one could approximate linear statistics of local real roots for $F_n$ by some linear statistics of local complex roots for $F_n$, and thus the desired real universality estimates follow  from  Theorem~\ref{thm:general:complex}.  

To formulate the observation, we will prove
\begin{lemma}\label{l.repulsion} If $c_2>0$ is sufficiently small, then uniform over $x\in I+(-c/n,c/n)$, we have
\begin{equation}\label{e.repulsion}
P(N(B(x,n^{-(1+c_2)})) \ge 2) =O(n^{-\frac 3 2 c_2}).
\end{equation}
\end{lemma}

We now go into the proof of Lemma~\ref{l.repulsion}. Let $c_2>0$ be small, and $\gamma=n^{-(1+c_2)}$. Let $N$ be the number of roots of $F_n$ in $B(x,\gamma)$, and let $\widetilde N$ be the number of roots for $\widetilde F_n$ in $B(x,1.01\gamma)$. Let $\varphi:\C\to[0,1]$ be a bump function supported inside $B(x,1.01\gamma)$ which is equal to $1$ on $B(x,\gamma)$.
We now use Theorem~\ref{thm:general:complex} (for $k=1,2$) for  $G(z)=n^{Cc_2}\varphi(z)$ and $G(z,w)=n^{Cc_2}\varphi(z)\varphi(w)$ for some large absolute constant $C$ (independent of $c_2$). It follows that, for some  constants  $c>0$ (independent of $c_2$), we have
\begin{eqnarray*}
P(N\ge 2) &\le& \E N (N -1) \\
&\ll& n^{-Cc_2} n^{-c}  + \E \widetilde N(\widetilde N-1)\\
&\ll&  n^{-Cc_2} n^{-c}  + \E [\widetilde N^2 1_{N \gg n^{c_1}}] + n^{2c_1}\P(\widetilde N\ge 2),
\end{eqnarray*}
where $c_1>0$ is a very small constant compared to $c_2$. Using   Lemma~\ref{cond-smallball}, and Lemma~\ref{cond-bddn} and Jensen's formula, we obtain
$$P(\widetilde N \ge C' n^{c_1}) \quad \le \quad  P(N_{\widetilde F_n}(B(x,\frac c n)) \ge C' n^{c_1})   \quad \ll \quad n^{-A}$$
for some $C'>0$ and any $ A>0$, both are independent of $c_1,c_2$.
Thus, using Lemma~\ref{cond-poly} we obtain
\begin{eqnarray*} 
  \E [\widetilde N^{2} 1_{\widetilde N \ge C' n^{c_1}}]   &\ll&  n^{-c'}.
\end{eqnarray*}
for some constants $c'>0$, $C'>0$ (both are independent of $c_1,c_2$). Thus, by choosing $c_2$ sufficiently small compared to $c$ and $c'$, we obtain
\begin{eqnarray*}
\P(N\ge 2) &\ll& n^{-\frac 3 2c_2} + n^{2c_1}\P(\widetilde N\ge 2).
\end{eqnarray*}
Since we choose $c_1$ very small compared to $c_2$, it suffices to prove that
$$ \P(\widetilde N\ge 2) \ll n^{-(\frac 3 2+\ep)c_2},$$
for some $\ep>0$. We'll show it for $\ep=0.1$.  Consider the Taylor polynomial 
$$\ell(z)=\widetilde F_n(x)+(z-x)\widetilde F_n'(x)$$ and let $\delta(z)=\widetilde F_n(z)-\ell(z)$. Since $\ell$ has at most one root, by Rouch\'e's theorem we have
$$\P(\widetilde N\ge 2) \le \P\left (\min_{|z-x|=2\gamma} |\ell(z)|\le \max_{|z-x|=2\gamma} |\delta(z)|\right ).$$
For convenience of notation, in the rest of the proof, let 
$$T= n^{-2c_2+c_1} (\sum_j |p_j(x)|^2)^{1/2}.$$
Now, using Lemma~\ref{cond-repulsion}, we have
$$\max_{|z-x|\le 4\gamma}  Var[\delta(z)] \ll \gamma^4 \sum_{j=1}^n \sup_{|w-x|\le 4\gamma} |p_j^{''}(w)|^2 \ll n^{-c_1}T^2.$$
Since $\delta(z)$ is Gaussian (for any fixed $z$) with zero mean, using convexity of the $\exp$ function and Jensen's inequality, we have
$$\E \left  [\exp\left (\frac 1{\gamma}\int_{|z-x|=4\gamma} \frac{|\delta(z)|}{\sqrt{Var[\delta(z)]}} ds\right ) \right ]\le \frac 1{\gamma}\int_{|z-x|=4\gamma}  \E \left  [\exp\left (\frac{|\delta(z)|}{\sqrt{Var[\delta(z)]}} ds\right )\right ] \ll 1$$
($\int \cdot ds$ is the usual line integral). Using Cauchy's integral formula, it follows that
$$\P\left (\max_{|z-x|=2\gamma} |\delta(z)|>T\right ) \ll \exp\left (-\frac{cT}{\max_{|z-x|=4\gamma} \sqrt{Var[\delta(z)]}}\right )  \ll \exp(-cn^{c_1/2})  \ll n^{-1.6c_2}.$$

On the other hand, using linearity of $\ell$ we have 
$$\min_{|z-x|=2\gamma}|\ell(z)| = \min(|\ell(x-2\gamma)|,|\ell(x+2\gamma)|).$$
Using Lemma~\ref{cond-repulsion} again and the fact that $c_1\ll c_2$, it follows that
$$Var[\ell(x+2\gamma)] = \sum_{j=1}^n (p_j(x)+2\gamma p'_j(x))^2 \gg \sum_j  p_j(x)^2 = T^2 n^{-2c_1+4c_2},$$
therefore by standard Gaussian estimates for any $c>0$, we have
$$\P(|\ell(x+2\gamma)|\le cT) \ll n^{-2c_2+c_1} \ll n^{-1.6c_2}.$$
A similar estimate holds for $\ell(x-2\gamma)$ by the same argument. Thus,
$$\P\left (\min_{|z-x|=2\gamma}|\ell(z)| < CT<\max_{|z-x|=2\gamma} |\delta(z)|\right ) \ll n^{-1.6c_2} ,$$
as desired. This completes the proof of Lemma~\ref{l.repulsion}. \endproof

\subsection{Step 3: Finish the proof}
Using Lemma \ref{l.repulsion}, the rest of the proof goes as follows.  Let $c_2>0$ be small, and $\gamma=n^{-(1+c_2)}$, and let $S_1$ be the rectangle
$$S_1:=[x_1-\frac c{n}, x_1+\frac c {n}]\times [-\frac {\gamma}{4}, \frac {\gamma}{4}].$$
We may cover $S_1$ using $O(n^{-1}/\gamma)=O(n^{c_2})$ balls of radius $\gamma$ with center inside $I+(-c/n,c/n)$. It follows that with high probability there are no roots of $F_n$ in $S_1\setminus \R$:
\begin{equation}\label{e.obs}
\P(N(S_1\setminus \R)\ge 1) = O(n^{-\frac 3 2 c_2} n^{c_2}) = O(n^{-c_2/2}).
\end{equation}
This is the observation we mentioned above. Now, let $\phi:\R \to [0,1]$ be smooth, supported on $[-1,1]$ with $\phi(0)=1$. Let $H(x+iy)=G_1(x)\phi(\frac{4y}{\gamma})$ be a function on $\C$. Recall that $\zeta_1,\dots,$ denote the complex roots of $F_n$. Let 
$$X=\sum_{j} H(\zeta_j), \ \ Y=\sum_{\zeta_j\in \R} G_1(\zeta_j).$$
Since universality holds for the complex roots of $F_n$ (Theorem~\ref{thm:general:complex}), it suffices to show that $\E |X -  Y| = O(n^{-c})$ for some $c>0$. We could actually show, for any $k\ge 1$,
\begin{eqnarray}\label{e.Nmoment}
\E |X -  Y|^k = O_k(n^{-c}),
\end{eqnarray}
for some $c>0$ that may depend on $k$. To see this, first note that $|X-Y|\le C N(S_1\setminus \R)$
for some $C>0$. Consequently, thanks to the observation \eqref{e.obs},
\begin{eqnarray*}
	\P(X\ne Y)   &\ll&  n^{-c_2/2}.
\end{eqnarray*}
Thus, using H\"older's inequality, to show \eqref{e.Nmoment} it suffices to show the weaker estimate
$$\E |X-Y|^k = O_k (1).$$
Note that $|X-Y|  \ll N(B(x_1,\frac c {n}))$. Using Jensen's bound, for $c>0$ sufficiently small it holds that
$$N\left (B\left (x_1,\frac c {n}\right )\right ) \ll \log  \max_{w\in B(x_1, \frac2n)} |F_n(w)| - \log \max_{w\in B(x_1, \frac 1{5n})} |F_n(w)|.$$
Using Lemma~\ref{cond-smallball} and Lemma~\ref{cond-bddn}, it follows that for some $C>0$ the following holds for all $A>0,c_1>0$ 
$$\P(|X-Y| > C n^{c_1}) \ll n^{-A}.$$
Also, using Lemma~\ref{cond-poly} there we have $|X-Y|=O(n)$.  
It follows that
\begin{eqnarray*}
	\E |X-Y|^k &\ll&    n^{k} \P(|X-Y|\ge C n^{c_1})  + n^{c_1k} \P(|X-Y|\ne 0)\\ 
	&\ll& 1
\end{eqnarray*}
by choosing $c_1$ sufficiently small (compared to $c_2$) and $A$ sufficiently large.  This completes the (conditional) proof of Theorem~\ref{thm:general:real}.

\section{Proof of Theorem \ref{thm:local:expectation} and Corollary \ref{cor:local:expectation}}\label{sec:pf:local}
\textit{Proof of Theorem \ref{thm:local:expectation}.}	Our idea is to apply Theorem \ref{thm:general:real} to the indicator function $\textbf{1}_{[a, b]}$. However, since the indicator function is not smooth, we will approximate it below and above by smooth functions, say $\hat G_1$ and $\hat G_2$. To this end, let $\alpha = n^{-1-c/12}$ where $c$ is the constant in Theorem \ref{thm:general:real} and let $\hat G_1, \hat G_2:\R \to \R$ be any functions with continuous sixth derivative and satisfy the following
\begin{itemize}
	\item $\textbf{1}_{[a-\alpha, b+\alpha]}\ge \hat G_2\ge \textbf{1}_{[a, b]}\ge \hat G_1\ge  \textbf{1}_{[a+\alpha, b-\alpha]}$ \quad pointwise,
	\item $||F^{a}_i||_{\infty}\ll \alpha^{-a} \ll n^{a+c/2}$ for all $i=1, 2$ and $1\le a\le 6$.
\end{itemize}
Applying Theorem \ref{thm:general:real} to $G_i = n^{-c/2} \hat G_i$ (where the factor $n^{-c/2}$ is added so that $||G^{a}_i||_{\infty}\ll n^{-a}$ as in the hypothesis), we obtain 
\begin{eqnarray} 
 \E\sum G_1\left (\zeta_{i} \right) 
-\E\sum G_1\left (\tilde \zeta_{i} \right) \ll n^{-c},\quad \E\sum G_2\left (\zeta_{i} \right) 
-\E\sum G_2\left (\tilde \zeta_{i} \right) \ll n^{-c}\nonumber
\end{eqnarray} 
which gives
\begin{eqnarray} \label{eq:F:12:0}
\E\sum \hat G_1\left (\zeta_{i} \right) 
-\E\sum \hat G_1\left (\tilde \zeta_{i} \right) \ll n^{-c/2}, \quad \E\sum \hat G_2\left (\zeta_{i} \right) 
-\E\sum \hat G_2\left (\tilde \zeta_{i} \right) \ll n^{-c/2}.
\end{eqnarray} 
We shall prove later that 
\begin{equation} \label{eq:F:12}
\E\sum  \hat G_1 \left (\tilde\zeta_{i} \right)  - \E\sum  \hat G_2\left (\tilde\zeta_{i} \right) \ll n^{-c/24}.
\end{equation}
Since $\sum \hat G_1 \left (\zeta_{i} \right)  \le N[a, b]\le \sum \hat G_2\left (\zeta_{i} \right)  $,  we conclude from \eqref{eq:F:12:0} and \eqref{eq:F:12} that 
\begin{equation} 
\E N[a, b] - \E \tilde N[a, b] \ll n^{-c/24}\nonumber
\end{equation}
as claimed.

To prove \eqref{eq:F:12}, we note that $0\le \hat G_2 - \hat G_1\le \textbf{1}_{[a-\alpha, a+\alpha]} + \textbf{1}_{[b-\alpha, b+\alpha]} $, thus it suffices to show the following
\begin{lemma}
	Let $I\subset K_{\ep'}$ be an interval of length $n^{-1-c/12}$. Then
	\begin{equation*}\label{key}
	\E \tilde N(I) \ll n^{-c/24}.
	\end{equation*}
\end{lemma}
\begin{proof}
	By the Kac-Rice formula \cite{Kac1943average} (also \cite[Proposition 1.1]{lubinsky2016linear}), we have 
	\begin{eqnarray}\label{key}
	\frac{1}{n}\E \tilde N(I) &=& \int_{I} \frac{1}{\pi} \frac{\tilde K_{n}(x, x)}{n}\sqrt{\frac{\tilde K_{n}^{(1,1)}(x, x)}{\tilde K_{n}(x, x)^{3}} - \left (\frac{\tilde K_{n}^{(0,1)}(x, x)}{\tilde K_{n}(x, x)^{2}}\right )^{2}} dx  \nonumber\\
	&\le&\int_{I} \frac{1}{\pi} \frac{\tilde K_{n}(x, x)}{n}\sqrt{\frac{\tilde K_{n}^{(1,1)}(x, x)}{\tilde K_{n}(x, x)^{3}}  } dx  \nonumber
	\end{eqnarray}
	where
	$$K_{n}^{(k, l)}(x, x) =  \sum_{i=0}^{n} p_i^{(k)}(x) p_i^{(l)}(x), \quad \tilde K_{n}^{(k, l)}(x, x) = w(x)\sum_{i=0}^{n} p_i^{(k)}(x) p_i^{(l)}(x).$$

By Condition \eqref{cond:w}, $\tilde K_{n}^{(k, l)}(x, x)$ has the same order as $K_n^{(k, l)}(x, x)$.

Thus, by Condition \eqref{cond:pn}, we have
	\begin{equation*}\label{key}
 \sup_{x\in I} \frac{1}{n}\tilde K_{n}(x, x) \ll \sup _{n\ge 1}\sup_{x\in I} \frac{1}{n}  K_{n}(x, x)  \ll 1
	\end{equation*}
and by \eqref{cond:kernel:2}, 
 \begin{equation*}\label{key}
 \inf_{x\in I}  \tilde K^{3}_{n}(x, x) \gg n^{3}.
 \end{equation*}
By \eqref{eq:bound:pj:1dev}, we have
\begin{equation*}\label{key}
 \sup_{x\in I}  \tilde K^{(1, 1)}_{n}(x, x) \ll n^{3}.
\end{equation*}
	Therefore,  
	\begin{eqnarray*}\label{key}
	\frac{1}{n}\E \tilde N(I) &\ll&  \int_{I}  1 dx \ll |I| = n^{-1-c/12}
	\end{eqnarray*}
	as desired.
	\end{proof}

\begin{proof} [Proof of Corollary \ref{cor:local:expectation}]
	This corollary follows by decomposing the interval $[a, b]$ into $1+n(b-a)$ disjoint intervals of length at most $1/n$. Applying Theorem \ref{thm:local:expectation} to each of these intervals and adding up, we obtain
	\begin{equation*}\label{key}
		\E N_n(a, b)- \E \tilde N_n(a, b) \ll (1+n(b-a))n^{-c}
	\end{equation*}
which gives the desired bound.
	\end{proof}

\section{Proof of Theorem \ref{thm:edge}} \label{sec:pf:edge}
We shall use the following beautiful result from \cite{dauvergne2021necessary}. Let us denote by $\nu_n$ the empirical measure of the roots $(\zeta_i)$ of $F_n$
$$\nu_n = \frac{1}{n}\sum_{i=1}^{n} \delta_{\zeta_i}.$$
We recall that $\nu_\K$ is the equilibrium measure of $\K$
\begin{theorem} (\cite[Theorem 1.1]{dauvergne2021necessary})\label{thm:conv:nu}
	Under the hypothesis of Theorem \ref{thm:edge}, the sequence of empirical measures $\nu_n$ converges weakly to $\nu_{\K}$ with probability 1.
\end{theorem}
Here, we note that the hypothesis of Theorem \ref{thm:edge} guarantees that the measure $\mu$ is regular.

 Using this theorem and apply the same reasoning as in the proof of \cite[Corollary 2.5]{lubinsky2016linear}, we obtain the following.
\begin{proposition} (\cite[Lemma 3.3]{lubinsky2016linear}) \label{thm:expectation:limit} Under the hypothesis of Theorem \ref{thm:edge}, let $E\subset \mathbb{C}$  be a compact set satisfying $\nu_{\K}(\partial E)=0$ where $\partial E$ is the boundary of $E$ in $\C$.  Then
$$
\lim_{n\rightarrow\infty}\frac{\E N_{n}(E)}{n}=\nu_{\K}(E).
$$
\end{proposition}

%

Recall $\K_{\ep'} = \{x\in \K: [x- \ep', x+\ep']\subset \K\}$. Let  $A_1 = \partial \K+[-\ep', \ep']$ and $A_2 = \R\setminus (\K_{\ep'}\cup A_1)$.

We shall apply Proposition \ref{thm:expectation:limit} to $E = \partial \K +B_\C(0, \ep')$ where $B_\C(0, \ep')$ is the closed disk in the complex plane of radius $\ep'$. Since $\partial E_r\cap \K = \{a\pm r: a\in \partial \K\}$ is a collection of finitely many points and since $\nu_\K$ is continuous, $\nu_\K(\partial E) = 0$. Applying Proposition \ref{thm:expectation:limit} to $E$ and then send $\ep'$ to $0$, we obtain
$$
\lim_{\ep'\to 0}\lim_{n\rightarrow\infty}\frac{\E N_{n}(E)}{n}=\lim_{\ep'\to 0} \nu_{\K}(E) = 0.
$$
Since $A_1\subset E$,  this yields
$$
\lim_{\ep'\to 0}\lim_{n\rightarrow\infty}\frac{\E N_{n}(A_1)}{n} = 0.
$$
 
For $A_2$, let $F$ be strip $(\K+[-\ep', \ep'])\times [-\ep', \ep']$ that covers $\K$. Since $\nu_{\K}$ is supported on $\K$, $\nu_{\K}(\partial F)=0$ for sufficiently small $\ep'$. And so, 
$$  \lim_{n\rightarrow\infty}\frac{\E N_{n}(F)}{n}=\nu_{\K}(F)=\nu_{\K}(\K)=1.$$ 

Subtracting both sides from $1$, we obtain 
$$  \lim_{n\rightarrow\infty}\frac{\E N_{n}(\mathbb{C}\backslash F)}{n}=0$$ 
which implies that $$  \lim_{n\rightarrow\infty}\frac{\E N_{n}(A_2)}{n}=0$$ as $A_2\subset \C\setminus F$. This completes the proof.

\section{Proof of Theorems \ref{t.main1} and \ref{thm:real}} \label{sec:pf:main}

\begin{proof}[Proof of Theorem \ref{t.main1}]  For any compact interval $[a, b]$ in the interior of $\K$, by Corollary \ref{cor:local:expectation}, we have
\begin{equation*}\label{key}
	\frac{\E N_n(a, b)}{n} - 	\frac{\E \tilde N_n(a, b)}{n} \ll (1/n+(b-a))n^{-c} = o(1). \nonumber
\end{equation*}
By \eqref{eq4}, $\frac{\E \tilde N_n(a, b)}{n}$ converges to $\frac{1}{\sqrt{3}}\nu_K ([a,b])$. And hence $\frac{\E N_n(a, b)}{n}$ converges to the same limit.
\end{proof}

 \begin{proof}[Proof of Theorem \ref{thm:real}] 
We now move on to proving \eqref{eq:main1:1}. For every $\delta>0$, by Theorem \ref{thm:edge}, there exist $\ep'>0$ and $N>0$ such that for all $n>N$, 
\begin{equation}\label{eq:e1}
\frac{\E N_n(\R\setminus \K_{\ep'})}{n} <\delta.
\end{equation}
In the special case when $\xi_i$ are Gaussian, Theorem \ref{thm:edge} also holds and so we also have (with possibly a different $N$ which we then only need to increase $N$)
\begin{equation}\label{eq:e2}
\frac{\E \tilde N_n(\R\setminus \K_{\ep'})}{n} <\delta,\quad \forall n>N. 
\end{equation}

Since $\K_{\ep'}$ is a union of finitely many subintervals, applying Corollary \ref{cor:local:expectation} to each subintervals in $\K_{\ep'}$ and adding up yield
\begin{equation*}\label{eq:e33}
	\frac{\E N_n(\K_{\ep'})}{n} - 	\frac{\E \tilde N_n(\K_{\ep'})}{n} = o(1).  
\end{equation*}

Combining \eqref{eq:e1}, \eqref{eq:e2} and \eqref{eq:e33}, we obtain 
\begin{equation*}\label{eq:e3}
	\left |\frac{\E N_n(\R)}{n} - 	\frac{\E \tilde N_n(\R)}{n} \right |\le 3\delta
\end{equation*}
for sufficiently large $n$. By \eqref{eq2}, $\frac{\E \tilde N_n(\R)}{n} $ converges to $\frac{1}{\sqrt 3}$. So does $\frac{\E N_n(\R)}{n} $.
\end{proof}

\section{Proof of Theorem \ref{t.main2}} \label{sec:pf:2}

For this proof, we shall use the following classical results. The first result gives sufficient conditions for $p_n$ to be uniformly bounded on the entire support. We use Theorem 5.2 in Table I, page 245 of \cite{freudorthogonal} where we also refer to page 229 of \cite{freudorthogonal} for the definition of the function space. The remark containing (5.26) in \cite[page 244]{freudorthogonal} explains how to prove the aforementioned Theorem 5.2 (page 245) using the original version of \cite[Theorem 5.2)]{freudorthogonal} in page 232.
\begin{theorem} \cite[Theorem 5.2, page 245]{freudorthogonal} \label{thm:freud:1}
	Let $\alpha$ be an absolutely continuous measure supported on $[-1, 1]$ with density $\alpha'$. Suppose that there exist constants $c$ and $C$ so that $0\le c\le \alpha'(x) \sqrt{1-x^{2}}\le C$ for all $x\in (-1, 1)$ and $\alpha'$ satisfies the following Lipschitz condition
	\begin{eqnarray}\label{eq:lipschitz}
		\int_{-\pi}^{\pi}\left (\alpha'(\cos(\theta+h)) |\sin(\theta+h)|-\alpha'(\cos(\theta)) |\sin(\theta)|\right )^{2} d\theta &=& O(h) \quad\text{as } h\to 0.
	\end{eqnarray}
	Then the sequence of orthonormal polynomials with respect to $\alpha$ is uniformly bounded on $[-1, 1]$. 
\end{theorem}

The next result, also from \cite{freudorthogonal} gives less stringent conditions for $p_n$ to be uniformly bounded in any compact subset of the interior of the spectrum. 
\begin{theorem}\cite[Theorem 5.7, page 245]{freudorthogonal} \label{thm:freud:2}
	Suppose that $\alpha$ and $\mu$ are finite measures that are absolutely continuous and supported on $[-1, 1]$ with densities $\alpha'$ and $\mu'$ respectively. Assume that for some constants $c$ and $C$, $0\le c\le \alpha'(x)\sqrt{1-x^{2}}\le C$ for all $x\in (-1, 1)$ and that the sequence of orthonormal polynomials with respect to $\alpha$ is uniformly bounded on $[-1, 1]$. Let $[a, b]\subset (-1, 1)$. Assume that $\alpha'(x) = \mu'(x)$ for all $x \in [a, b]$ and  
	\begin{eqnarray*}\label{cond:freud:2}
		\int_{-1}^{1} \frac{dx}{\mu'(x)\sqrt{1-x^{2}}}<\infty.
	\end{eqnarray*}
	Then the orthonormal sequence of $\mu$ is uniformly bounded in every closed part of $(a, b)$. 
\end{theorem}

We are now ready to prove Theorem \ref{t.main2}.

 	For every $\ep'>0$, let $I = [-1+\ep', 1-\ep']$ be a compact interval in $(-1, 1)$. All we need to do is to show that \eqref{cond:pn} holds for $I$.
 	
Let $\alpha$ be the measure supported on $[-1, 1]$ with density
	$$\alpha'(x) = w(x)\textbf{1}_{x\in [-1 +\ep'/3, 1-\ep'/3]}+ \frac{1}{\sqrt{1-x^{2}}} \textbf{1}_{x\in [-1, 1]\setminus [-1 +\ep'/3, 1-\ep'/3]}.$$
	
	Observe that $\alpha' = \mu'$ on $[-1+\ep'/3, 1 - \ep'/3]$ and $0\le c\le \alpha'(x)\sqrt{1-x^{2}}\le C$ for all $x\in (-1, 1)$ for some constant $c, C$.
	
	We shall now show that $\alpha'$ satisfies \eqref{eq:lipschitz}. To this end, letting 
	$$E = \{\theta\in [-\pi, \pi]: \cos \theta \in [-1+\ep'/3, 1-\ep'/3] \},$$ 
	we decompose the interval $[-\pi, \pi]$ into 
	$A_h = \{\theta: \theta\in E, \theta + h\in E\}$, $B_h = \{\theta: \theta\notin E, \theta+h\notin E\}$, and $C_h = [-\pi, \pi]\setminus (A_h\cup B_h)$.
	
	Observe that on $A_h$, $\alpha'(\cos(\theta+h)) |\sin(\theta+h)|-\alpha'(\cos(\theta)) |\sin(\theta)| = w(\cos(\theta+h)) |\sin(\theta+h)|-w(\cos(\theta)) |\sin(\theta)|$. By Condition \eqref{cond:ls}, 
	\begin{eqnarray*} 
		\int_{A_h}\left (\alpha'(\cos(\theta+h)) |\sin(\theta+h)|-\alpha'(\cos(\theta)) |\sin(\theta)|\right )^{2} d\theta &=& O(h) \quad\text{as } h\to 0.
	\end{eqnarray*}
	
	On $B_h$, $\alpha'(\cos(\theta+h)) |\sin(\theta+h)|-\alpha'(\cos(\theta)) |\sin(\theta)| = 1- 1 = 0$, so the corresponding integral on $B_h$ is 0. Finally, observe that $|C_h| = O(h)$ and that for all $t\in [-\pi, \pi]$, we have
	$$|\alpha'(\cos(t))| |\sin(t)| \le \min\{1 , |w(\cos t)|\textbf{1}_{\cos t  \in [-1 +\ep'/3, 1-\ep'/3]}\} = O(1).$$ 
	Thus, the integral on $C_h$ is also of order $\int_{C_h} 1 = O(h)$.
	
	Combining these bounds, we obtain \eqref{eq:lipschitz}.
	
	Having proved that $\alpha $ satisfies the assumptions of Theorem \ref{thm:freud:1}, we conclude that the sequence of orthonormal polynomials with respect to $\alpha$ is uniformly bounded on $[-1, 1]$.  By Assumption \eqref{cond:cir}, $\mu$ also satisfies the hypothesis of Theorem \eqref{thm:freud:2}. Thus, by this theorem, 
	we obtain 
	\begin{equation*}\label{key}
	|p_j(x)| = O(1) \quad\text{uniformly for all $x\in [-1+\ep'/2, 1 - \ep'/2]$ and all $j$,}
	\end{equation*}
as desired.

 \section{Acknowledgements} We would like to thank Igor Pritsker and Doron Lubinsky for helpful discussions and references.
 

\bibliographystyle{plain}
\bibliography{polyref}
\end{document}